\documentclass[bibother]{asl}
\usepackage{rotate}

\usepackage[UKenglish]{babel}
\usepackage[utf8]{inputenc}
\usepackage{verbatim}

\usepackage{theoremref}
\usepackage{hyperref}

\usepackage{dsfont}
\usepackage{enumitem}

\usepackage{thmtools}

\usepackage{catoptions}
\makeatletter

\def\Autoref#1{%
	\begingroup
	\edef\reserved@a{\cpttrimspaces{#1}}%
	\ifcsndefTF{r@#1}{%
		\xaftercsname{\expandafter\testreftype\@fourthoffive}
		{r@\reserved@a}.\\{#1}%
	}{%
		\ref{#1}%
	}%
	\endgroup
}
\def\testreftype#1.#2\\#3{%
	\ifcsndefTF{#1autorefname}{%
		\def\reserved@a##1##2\@nil{%
			\uppercase{\def\ref@name{##1}}%
			\csn@edef{#1autorefname}{\ref@name##2}%
			\autoref{#3}%
		}%
		\reserved@a#1\@nil
	}{%
		\autoref{#3}%
	}%
}
\makeatother

\theoremstyle{plain}
\newtheorem{theorem}{Theorem}[section]
\newtheorem{corollary}[theorem]{Corollary}
\newtheorem{lemma}[theorem]{Lem\-ma}
\newtheorem{proposition}[theorem]{Prop\-o\-si\-tion}
\newtheorem{question}[theorem]{Question}
\newtheorem{conjecture}[theorem]{Conjecture}

\newtheorem{fact}[theorem]{Fact}

\theoremstyle{definition}
\newtheorem{definition}[theorem]{Definition}
\newtheorem{construction}[theorem]{Construction}
\newtheorem{remark}[theorem]{Remark}
\newtheorem{example}[theorem]{Example}

\newcommand{\N}[0]{\mathbb{N}}
\newcommand{\Z}[0]{\mathbb{Z}}
\newcommand{\Q}[0]{\mathbb{Q}}
\newcommand{\R}[0]{\mathbb{R}}
\renewcommand{\P}[0]{\mathbb{P}}
\renewcommand{\H}[0]{\mathbb{H}}
\newcommand{\OO}[0]{\mathcal{O}}
\newcommand{\MM}[0]{\mathcal{M}}
\newcommand{\NN}[0]{\mathcal{N}}
\newcommand{\supp}[0]{\mathrm{supp}}

\newcommand{\brackets}[1]{\left( #1 \right)}
\newcommand{\setbr}[1]{\left\{ #1 \right\}}
\newcommand{\pow}[1]{\!\left(\!\left( #1 \right)\!\right)}

\renewcommand{\L}[0]{\mathcal{L}}
\newcommand{\Lor}[0]{\mathcal{L}_{\mathrm{or}}}
\newcommand{\Log}[0]{\mathcal{L}_{\mathrm{og}}}
\newcommand{\Lr}[0]{\mathcal{L}_{\mathrm{r}}}
\newcommand{\Lvf}[0]{\mathcal{L}_{\mathrm{vf}}}

\newcommand{\Trcf}[0]{T_{\mathrm{rcf}}}
\newcommand{\Tdoag}[0]{T_{\mathrm{doag}}}
\newcommand{\rc}[1]{{#1}^{\mathrm{rc}}}
\renewcommand{\div}[1]{{#1}^{\mathrm{div}}}

\newcommand{\ol}[1]{\overline{#1}}
\newcommand{\ul}[1]{\underline{#1}}
\newcommand\restr[2]{{
		\left.\kern-\nulldelimiterspace 
		#1
		\vphantom{\big|} 
		\right|_{#2}
}}
\newcommand{\vmin}[0]{v_{\min}}
\newcommand{\vnat}[0]{v_{\mathrm{nat}}}

\DeclareMathOperator{\cl}{cl}
\DeclareMathOperator{\dcl}{dcl}
\DeclareMathOperator{\Th}{Th}

\DeclareMathOperator{\ff}{ff}
\DeclareMathOperator{\td}{td}

\usepackage{xcolor}

\def\confname{}
\def\copyrightyear{2000}

\title[Ordered fields dense in their real closure]{Ordered fields dense in their real closure and definable convex valuations}

\thanks{We started this research at the \emph{Model Theory, Combinatorics and Valued fields Trimester} at the Institut Henri Poincaré in March 2018. All three authors wish to thank the IHP for its hospitality, and Immanuel Halupczok, Franziska Jahnke, Vincenzo Mantova and Florian Severin for discussions. We also thank Assaf Hasson for helpful comments on a previous version of this work and Vincent Bagayoko for giving an answer to a question leading to \Autoref{prop:ordfieldnoarchmodel}. Moreover, we thank the anonymous referee for providing several helpful comments and references.}

\author[L.~S.~Krapp]{Lothar Sebastian Krapp}
\revauthor{Krapp, Lothar Sebastian}
\author[S.~Kuhlmann]{Salma Kuhlmann}
\revauthor{Kuhlmann, Salma}
\author[G.~Lehéricy]{Gabriel Lehéricy}
\revauthor{Lehéricy, Gabriel}

\address{Fachbereich Mathematik und Statistik\\Universität Konstanz\\78457 Konstanz, Germany}
\email{sebastian.krapp@uni-konstanz.de}
\urladdr{http://www.math.uni-konstanz.de/\urltilde krapp/}
\thanks{The first author was supported by a doctoral scholarship of Studienstiftung des deutschen
	Volkes as well as of Carl-Zeiss-Stiftung { and an Independent Research Grant of Zukunftskolleg, Universität Konstanz}.}

\address{Fachbereich Mathematik und Statistik\\Universität Konstanz\\78457 Konstanz, Germany}
\email{salma.kuhlmann@uni-konstanz.de}
\urladdr{https://www.mathematik.uni-konstanz.de/kuhlmann/}

\address{Fachbereich Mathematik und Statistik\\Universität Konstanz\\78457 Konstanz, Germany}
\curraddr{École supérieure d'ingénieurs Léonard-de-Vinci\\Pôle Universitaire Lé\-o\-nard de Vinci\\92 916 Paris La Défense Cedex, France}
\email{gabriel.lehericy@uni-konstanz.de}
\urladdr{http://www.math.uni-konstanz.de/\urltilde lehericy/}

\begin{document}

	\begin{abstract}
		In this paper, we undertake a systematic model and valuation theoretic study of the class of ordered fields which are dense in their real closure. We apply this study to determine definable henselian valuations on ordered fields, in the language of ordered rings.  { In light of our results, we re-examine the Shelah-Hasson Conjecture (specialised to ordered fields) and provide an example limiting its valuation theoretic conclusions.}
	\end{abstract}

\maketitle

	\section{Introduction}

	Let $\Lr = \{+,-,\cdot,0,1\}$ be the language of rings and $\Lor = \Lr \cup \{<\}$  the language of ordered rings.	
	There is a vast collection of results giving conditions on $\Lr$-definability\footnote{Throughout this paper, \emph{definable} means \emph{definable with parameters}.} of henselian valuations in fields, many of which are from recent years (cf.\ e.g.\ \cite{delon, hong, jahnke4, prestel, jahnke5}). {A survey on $\Lr$-definability of henselian valuations is given in \cite{fehm}.} In this paper, we undertake a systematic study of definable henselian valuations in ordered fields, which enables us to consider definability in the richer language $\Lor$. It is natural to expect from such a study to obtain strengthenings of existing $\Lr$-definability results. Ordered fields have been considered valuation theoretically mainly with respect to convex valuations. Note that due to \cite[Lemma~2.1]{knebusch}, henselian valuations in ordered fields are always convex. 
	A general method for the construction of non-trivial $\Lor$-definable convex valuations in ordered fields which are not dense in their real closure is given in \cite[Proposition~6.5]{jahnke}. By careful analysis of this method, we obtain sufficient conditions on the residue field and the value group of a henselian valuation $v$ in order that $v$ is $\Lor$-definable. These conditions motivate the study of the following two classes of structures: ordered abelian groups which are dense in their divisible hull, and ordered fields which are dense in their real closure. The former coincides with the elementary class of regular densely ordered abelian groups. We study the latter systematically, both from a model theoretic and an algebraic point of view. { Finally, we apply our results on definable valuations to strongly NIP\footnote{ The property ``strongly NIP'' is often also called ``strongly dependent''.} ordered fields and hence provide a counterexample to a strengthening of the Shelah--Hasson Conjecture (specialised to ordered fields). This is of particular interest in light of W.~Johnson's recent work \cite{johnson2}, in which the Shelah--Hasson Conjecture has been verified for dp-finite fields.}
		
	The structure of this paper is as follows.
	In \Autoref{sec:prelim}, we gather some basic preliminaries on ordered and valued abelian groups and fields. In \Autoref{sec:densdivhull}, we study the class of ordered abelian groups which are dense in their divisible hull.  Noting that it coincides with the class of regular densely ordered abelian groups (see \Autoref{prop:regimpliesdense}), we present a recursive elementary axiomatisation (see \Autoref{cor:regaxiom}). Moreover, in \Autoref{prop:subgroupdef} we obtain a characterisation of this class in terms of $\Log$-definability of convex subgroups (where $\Log = \{+,0,<\}$ is the language of ordered groups). Analogously, in \Autoref{sec:density} we study the class of ordered fields which are dense in their real closure, first model theoretically in \Autoref{sec:modeltheorydensity}. Our main result is \Autoref{prop:omindenserec}, {  which gives us a general method to produce recursive axiomatisations of ordered structures dense in certain definable closures. In particular, \Autoref{prop:omindenserec} leads to both an axiomatisation of ordered abelian groups which are dense in their divisible hull as well as an axiomatisation of ordered fields which are dense in their real closure (see \Autoref{cor:denserec})}. In \Autoref{sec:denseimm}, we briefly review the connections between dense and immediate extensions of ordered fields (see \Autoref{cor:denseimmfields}) and complements to the valuation ring (see \Autoref{fact:complements}). In \Autoref{sec:ip}, we relate normal integer parts of ordered fields to dense subfields regular over $\Q$ (see \Autoref{rmk:normalip}, \Autoref{cor:regularext}). Finally, we obtain a characterisation of density in real closure in terms of absolute transcendence bases (see \Autoref{cor:lemmaerdos}). In \Autoref{sec:defconvval}, we study $\Lor$-definable henselian valuations in ordered fields. In \Autoref{subsec:def}, we obtain our main result \Autoref{thm:defval} and compare it to known conditions for $\Lr$-definability of specific henselian valuations.  Special emphasis is put on definable valuations in almost real closed fields in \Autoref{sec:arc}. In \Autoref{sec:snip}, we relate our $\Lor$-definability results of henselian valuations in ordered fields to open questions regarding strongly NIP ordered fields.\footnote{A preliminary version of this work is contained in our arXiv preprint \cite{krapp}, which contains also a systematic study of the class of strongly NIP almost real closed fields. This study, being of independent interest, will be the subject of the separate publication \cite{krapp2}.}
	{ We exploit an equivalent reformulation of the Shelah--Hasson Conjecture specialised to ordered fields (see \autoref{conj:snipordered}) and examine further valuation theoretic strengthenings in the conclusion of this conjecture (see \autoref{qu:strengthening}). In this regard, we construct a strongly NIP almost real closed field which fails to be almost real closed with respect to an $\Lor$-definable valuation (see \autoref{thm:negativeanswer}).} We conclude by motivating and collecting open questions inspired by results throughout this paper in \Autoref{sec:questions}.
	
	\section{Preliminaries on ordered abelian groups and valued fields}\label{sec:prelim}
	
	We denote by $\N_0$ the set of natural numbers with $0$ and by $\N$ the set of natural numbers without $0$. All notions on valued fields and groups can be found in \cite{kuhlmann,engler}. 
	Throughout this work, we abbreviate the $\Lr$-structure of a field $(K,+,-,\cdot,0,1)$ by $K$, the $\Lor$-structure of an ordered field $(K,+,-,\cdot,0,1, <)$ by $(K,<)$ and the $\Log$-structure of an ordered group $(G,+,0, <)$ by $G$.
	
	Let $G \subseteq H$ be an extension of ordered abelian groups. We say that $G$ is \textbf{dense} in $H$ or that the extension $G \subseteq H$ is dense if for any $a,b \in H$ { with $a<b$} there is $c \in G$ such that $a<c<b$.
	If $\Z$ is a convex subgroup of $G$, then we say that $G$ is \textbf{discretely ordered}. Otherwise, $G$ is \textbf{densely ordered}, i.e.\ $G$ is dense in itself. 
	{ We call two elements $a,b \in G$ \textbf{archimedean equivalent} (in symbols $a \sim b$) if there is some $n\in \N$ such that { $|a|\leq n|b|$ and $|b| \leq n|a|$}. Let $\Gamma = \{[a] \mid a \in G\setminus\{0\}\}$, the set of archimedean equivalence classes of $G\setminus\{0\}$. Equipped with addition $[a]+[b] = [ab]$ and the ordering $[a] < [b] : \Leftrightarrow a \not\sim b \wedge |b| < |a|$, the set $\Gamma$ becomes an ordered set. Then $G\setminus\{0\} \to \Gamma, a \mapsto [a]$ defines a convex valuation on $G$. This is called the \textbf{natural valuation} on $G$, denoted by $v$.\footnote{See \cite[{ pages 9 and 15}]{kuhlmann} for further details regarding the definition of the natural valuation.} The value set $\Gamma$ of $G$ is denoted by $vG$.}
	We say that an extension of ordered abelian groups $G \subseteq H$ is \textbf{immediate} if it is immediate with respect to the natural valuation.\footnote{See \cite[page~3]{kuhlmann} for a definition of an immediate extension.}
	Let $\gamma \in vG$, and let $G^\gamma$ and $G_\gamma$ be the following convex subgroups of $G$: $$G^\gamma = \setbr{g \in G \mid v(g) \geq \gamma} \text{ and } G_\gamma = \setbr{g \in G \mid v(g) > \gamma}.$$
	The \textbf{archimedean component} $B(G,\gamma)$ of $G$ corresponding to $\gamma$ is given by $B(G,\gamma) = G^\gamma/G_\gamma$. If no confusion arises, we only write $B_\gamma$ instead of $B(G,\gamma)$. Note that $B_\gamma$ is an ordered abelian group with the order induced by $G$.
	A valuation $w$ on $G$ is \textbf{convex} if for any $g_1,g_2 \in G$ with $0<g_1\leq g_2$, we have $w(g_1) \geq  w(g_2)$. Note that there is a one-to-one correspondence between non-trivial convex subgroups of $G$ and final segments of $vG$ (cf.\ \cite[page~50~f.]{kuhlmann}).
	
	We also consider $G$ as a topological subspace of $H$ under the order topology. We say that $G$ has a \textbf{left-sided limit point} $g_0$ in $H$ if for any  $g_1 \in H$ with $g_1>0$ the intersection of $(g_0-g_1,g_0)$ with $G$ is non-empty. Similarly, $g_0$ is a \textbf{right-sided limit point} if for any $g_1 \in H^{>0}$ we have $(g_0,g_0+g_1) \cap G \neq \emptyset$. A \textbf{limit point} is a point which is a left-sided or a right-sided limit point.
	The \textbf{divisible hull} of $G$ is denoted by $\div{G}$. Note that $G$ and $\div{G}$ have the same value set under $v$, i.e.\ $vG = v\div{G}$. The \textbf{closure} of $G$ in $\div{G}$ with respect to the order topology is denoted by $\cl(G)$. Note that $G$ is dense in $\div{G}$ if and only if $\cl(G) = \div{G}$. Note further that $G$ has a limit point in $\div{G} \setminus G$ (i.e.\ there is some $a \in \div{G} \setminus G$ such that $a$ is a limit point of $G$ in $\div{G}$) if and only if $G$ is not closed in $\div{G}$.
	
	For any ordered abelian groups $G_1$ and $G_2$, we denote the \textbf{lexicographic sum} of $G_1$ and $G_2$ by $G_1 \oplus G_2$. This is the abelian group $G_1 \times G_2$ with the lexicographic ordering $(a,b) < (c,d)$ if $a<c$, or $a=c$ and $b<d$. Let $(\Gamma,<)$ be an ordered set and for each $\gamma \in \Gamma$, let $A_\gamma \neq \{0\}$ be an archimedean ordered abelian group. For any element $s$ in the product group $\prod_{\gamma \in \Gamma} A_\gamma$, define the \textbf{support of $s$} by $\supp(s) = \{\gamma \in \Gamma \mid s(\gamma) \neq 0\}$. The \textbf{Hahn product} $\H_{\gamma \in \Gamma}A_\gamma$ is the subgroup of $\prod_{\gamma \in \Gamma} A_\gamma$ consisting of all elements with well-ordered support. Moreover, $\H_{\gamma \in \Gamma}A_\gamma$ becomes an ordered group under the order relation $s > 0 :\Leftrightarrow s(\min\supp (s)) > 0$. We express elements $s$ of $\H_{\gamma \in \Gamma}A_\gamma$ by $s = \sum_{\gamma \in \Gamma} s_\gamma \mathds{1}_\gamma$, where $s_\gamma = s(\gamma)$ and $\mathds{1}_\gamma$ is the characteristic function of $\gamma$ mapping $\gamma$ to { some fixed $1_\gamma \in A_\gamma$ with $1_\gamma > 0$} and everything else to $0$. The \textbf{Hahn sum} $\coprod_{\gamma \in \Gamma} A_\gamma$ is the ordered subgroup of $\H_{\gamma \in \Gamma}A_\gamma$ consisting of all elements with finite support.\\
	
	Let $K$ be a field and let $v$ be a valuation on $K$. We denote the \textbf{valuation ring} of $v$ in $K$ by $\OO_v$, the \textbf{valuation ideal}, i.e.\  the maximal ideal of $\OO_v$, by $\MM_v$, the \textbf{ordered value group} by $vK$ and the \textbf{residue field} $\OO_v/\MM_v$ by $Kv$. For $a \in \OO_v$ we also denote $a + \MM_v$ by $\ol{a}$. For an ordered field $(K,<)$, a valuation is called \textbf{convex} (in $(K,<)$) if the valuation ring $\OO_v$ is a convex subset of $K$. In this case, the relation $\ol{a} < \ol{b} : \Leftrightarrow \ol{a} \neq \ol{b} \wedge a < b$ defines an order relation on $Kv$ making it an ordered field. 
	
	Let $\Lvf = \Lr\cup \{\OO_v\}$ be the \textbf{language of valued fields}, where $\OO_v$ stands for a unary predicate. Let $(K,\OO_v)$ be a valued field. An atomic formula of the form $v(t_1) \geq v(t_2)$, where $t_1$ and $t_2$ are $\Lr$-terms, stands for the $\Lvf$-formula $t_1=t_2=0 \vee (t_2\neq 0 \wedge \OO_v(t_1/t_2))$. Thus, by abuse of notation, we also denote the $\Lvf$-structure $(K,\OO_v)$ by $(K,v)$. Similarly, we also call $(K,<,v)$ an ordered valued field. 
	We say that a valuation $v$ is $\L$-definable for some language $\L \in \{\Lr,\Lor\}$ if its valuation ring is an $\L$-definable subset of $K$.
	
	Let $K$ be a field and let $v$ and $w$ be valuations on $K$. We write $v\leq w$ if and only if $\OO_v \supseteq \OO_w$. In this case we say that $w$ is \textbf{finer} than $v$ and $v$ is \textbf{coarser} than $v$. If $\OO_v \supsetneq \OO_w$, we write $v< w$ and say that $w$ is \textbf{strictly finer} than $v$ and that $v$ is \textbf{strictly coarser} than $w$. Note that $\leq$ defines an order relation on the set of convex valuations of an ordered field. 
	{ Since $(K,+,0,<)$ is an ordered abelian group, the notion of archimedean equivalence classes and the natural valuation specialise to the ordered field $K$. We denote the natural valuation on $K$ by $\vnat$.}
	If not further specified, we say that an extension of ordered fields $(K,<) \subseteq (L,<)$ is \textbf{immediate} if it is immediate\footnote{For the definition of an immediate extension of valued fields, see \cite[page~27]{kuhlmann}.} with respect to the natural valuation. The extension is \textbf{dense} if $K$ is dense in $L$.
	
	Let $(k,<)$ be an ordered field and $G$ an ordered abelian group. We denote the \textbf{ordered Hahn field} with coefficients in $k$ and exponents in $G$ by $k\pow{G}$. We denote an element $s \in k\pow{G}$ by $s = \sum_{g \in G} s_gt^g$, where $s_g = s(g)$ and $t^g$ is the characteristic function on $G$ mapping $g$ to $1$ and everything else to $0$. The ordering on $k\pow{G}$ is given by $s > 0 : \Leftrightarrow s(\min \supp s) > 0$, where $\supp s = \{g \in G \mid s(g) \neq 0\}$ is the \textbf{support} of $s$. Let $\vmin$ be the valuation on $k\pow{G}$ given by $\vmin(s) = \min \supp s$ for $s \neq 0$. Note that $\vmin$ is convex and henselian. Note further that if $k$ is archimedean, then $\vmin$ coincides with $\vnat$.
	
	We  repeatedly use the Ax--Kochen--Ershov principle for ordered fields (cf.\ \cite[Corollary~4.2(iii)]{farre}).
	
	\begin{fact}[Ax--Kochen--Ershov Principle]\thlabel{fact:ake}
		Let $(K, <,v)$ and $(L,<,w)$ be two ordered henselian valued fields. Then $(Kv,<) \equiv (Lw,<)$ and $vK \equiv wL$ if and only if $(K,<, v) \equiv (L,<, w)$.
	\end{fact} 

	\section{Regular densely ordered abelian groups}\label{sec:densdivhull}
	
	In this section, we study the class of ordered abelian groups which are dense in their divisible hull. In the light of ordered fields dense in their real closure and definable henselian valuations, this class is of interest for the following reasons:
	
	\begin{itemize}[wide, labelwidth=!, labelindent=6pt]
		\item The divisible hull $\div{G}$ of a non-trivial ordered abelian group $G$ is the smallest extension of $G$ which is a model of the theory of divisible ordered abelian groups $\Tdoag$. The real closure $(\rc{K},<)$ of an ordered field $(K,<)$ is the smallest extension of $(K,<)$ which is a model of the theory of real closed fields $\Trcf$. Thus, the divisible hull of an ordered abelian group can be seen as the group analogue to the real closure of an ordered field. Moreover, both $\Tdoag$ and $\Trcf$ share several model theoretic properties such as completeness, o-minimality and quantifier elimination. We are therefore interested in exploring algebraic and model theoretic similarities and differences between these two classes of ordered structures.
		
		\item The property of an ordered abelian group to be dense in its divisible hull can also be considered as a topological characteristic with respect to the order topology. To this end, for an ordered henselian valued field $(K,<,v)$, topological conditions on $vK$ as a subspace of $\div{vK}$ naturally arise in the context of the definability of $v$ (see \Autoref{thm:defval} and \Autoref{cor:hong}). 
	\end{itemize}

	Throughout this section, we prove properties of ordered abelian groups dense in their divisible hull and point out the analogues for ordered fields in \Autoref{sec:density}.

	Regular ordered abelian groups have been studied model theoretically in \cite{robinson} and algebraically in \cite{zakon}. An ordered abelian group is \textbf{regular} if it satisfies one of the equivalent conditions in \Autoref{fact:defregular} ({ cf.\  \cite[Theorem~2.1]{zakon} and \cite[Proposition~1 \& Proposition~4]{conrad}}).
	
	\begin{fact}\thlabel{fact:defregular}
		Let $G$ be an ordered abelian group. Then the following are equivalent:
		\begin{enumerate}[wide, labelwidth=!, labelindent=6pt]
			\item \thlabel{fact:defregular1} For any prime $p \in \N$ and for any infinite convex subset $A \subseteq G$, there is a $p$-divisible element in $A$.
			
			\item \thlabel{fact:defregular2} For any $n\in \N$ and any $a,b \in G$, if there are $g_1,\ldots,g_n \in G$ with $a \leq g_1 < \ldots < g_n \leq b$, then there  is some $c \in G$ with $a \leq nc \leq b$.
			
			\item For any non-trivial convex subgroup $H \subseteq G$, the quotient group $G/H$ is divisible.
		\end{enumerate}
	\end{fact}

	The following characterisation of regular densely ordered abelian groups is given in \cite[Theorem~2.1.~(a)]{zakon}.

	\begin{fact}\thlabel{prop:regimpliesdense}
		Let $G$ be an ordered abelian group. Then $G$ is dense in $\div{G}$ if and only if it is regular and densely ordered.
	\end{fact}
	
		

	\Autoref{prop:regimpliesdense} enables us to apply known results on regular densely ordered abelian groups to  ordered abelian groups which are dense in their divisible hull. 
	
	\begin{corollary}\thlabel{cor:regaxiom}
		There is a recursive $\Log$-theory { axiomatising} the class of non-trivial ordered abelian groups which are dense in their divisible hull.
	\end{corollary}

	\begin{proof}
		Let $\Sigma$ be the $\Log$-theory consisting of the axioms for non-trivial densely ordered abelian groups plus the collection of the following sentences (one for each $n \in \N$):
		$$ \forall a, b \ (\exists g_1,\ldots,g_n  \ a \leq g_1 < \ldots < g_n \leq b \to \exists c\ a \leq nc \leq b).$$
		By \Autoref{fact:defregular}~(\ref{fact:defregular2}), $\Sigma$ axiomatises the class of non-trivial regular densely ordered abelian groups, and thus also the class of non-trivial ordered abelian groups which are dense in their divisible hull by \Autoref{prop:regimpliesdense}.
	\end{proof}

	In \Autoref{cor:denserec}~\eqref{cor:denserec:2} we will see the analogue of \Autoref{cor:regaxiom} for the class of ordered fields which are dense in their real closure.
	We deduce further properties of ordered abelian groups which are dense in their divisible hull from the following fact, which is due to \cite{robinson} and \cite{zakon}.
	
	\begin{fact}\thlabel{fact:regrobinson}
		Let $G$ be an ordered abelian group. Then the following hold:
		
		\begin{enumerate}[wide, labelwidth=!, labelindent=6pt]
			\item \thlabel{fact:regrobinson:1} $G$ is regular if and only if it has an archimedean model (i.e.\ there is some archimedean ordered abelian group $H$ such that $H \equiv G$).
			
			\item $G$ is regular and discretely ordered if and only if it is a $\Z$-group (i.e.\ $G \equiv \Z$ as ordered groups).
		\end{enumerate}
	\end{fact}

	
	
	\begin{corollary}\thlabel{rmk:denseindivregular}
		Let $G$ be an ordered abelian group. Then $G$ is dense in $\div{G}$ if and only if $G$ has a densely ordered archimedean model.
	\end{corollary}

	\begin{proof}
		Both directions follow immediately from \Autoref{prop:regimpliesdense} and {   \Autoref{fact:regrobinson}\linebreak (\ref{fact:regrobinson:1})}.
	\end{proof}

	We will see in \Autoref{thm:densitytransfers2} that the analogue of the backward direction of \Autoref{rmk:denseindivregular} also holds for ordered fields, i.e.\ that any ordered field which has an archimedean model is dense in its real closure. However, the analogue of the forward direction of \Autoref{rmk:denseindivregular} does not hold for ordered fields, as there is an ordered field which is dense in its real closure but has no archimedean model (see \Autoref{prop:ordfieldnoarchmodel}). 
	
	In \Autoref{sec:defconvval}, we address the question what convex valuations are $\Lor$-definable in ordered fields. In analogy to \Autoref{constr:val}, we will show in \Autoref{prop:subgroupdef} that for any densely ordered abelian group which is not dense in its divisible hull, there exists a proper non-trivial convex $\Log$-definable subgroup. The following lemma is a useful characterisation of densely ordered abelian groups and will be applied several times.
	
	\begin{lemma}\thlabel{prop:denseeq}
		Let $G$ be an ordered abelian group. Then the following are equivalent:
		
		\begin{enumerate}[wide, labelwidth=!, labelindent=6pt]
			\item \thlabel{prop:denseeq:1} $G$ is densely ordered.
			
			\item \thlabel{prop:denseeq:2} $0$ is a limit point of $G$ in $G$.
			
			\item \thlabel{prop:denseeq:3} $0$ is a limit point of $G$ in $\div{G}$.
			
			\item  \thlabel{prop:denseeq:4} Either $vG$ has no last element, or $vG$ has a last element $\gamma$ and $B_\gamma$ is densely ordered.
		\end{enumerate} 
	\end{lemma}
	
	\begin{proof}
		We may assume that $G \neq \{0\}$, as in the case $G = \{0\}$ the e\-quiv\-a\-lences are clear.
		The equivalence of (\ref{prop:denseeq:1}) and (\ref{prop:denseeq:2}) is an easy consequence of the definition of a dense order on an abelian group. 
		Obviously, (\ref{prop:denseeq:3}) implies (\ref{prop:denseeq:2}). In \cite[Lemma~2.3]{zakon} it is shown that (\ref{prop:denseeq:1}) implies (\ref{prop:denseeq:3}).
		
		It remains to show that (\ref{prop:denseeq:1}) and (\ref{prop:denseeq:4}) are equivalent. Suppose that $vG$ has a last element $\gamma$ and $B_\gamma$ is not densely ordered. By (\ref{prop:denseeq:2}) applied to $B_\gamma$, there is some $g \in G$ with $g>0$ and $v(g) = \gamma$ such that there is no element in $B_\gamma$ strictly between $0 + G_\gamma$ and $g + G_\gamma$. Let $h\in G$ such that $0 < h \leq g$. Then $v(h) \geq v(g) = \gamma$. By maximality of $\gamma$, we have $v(h) = \gamma$. Thus, $0 + G_\gamma < h + G_\gamma \leq g + G_\gamma$. This implies $h + G_\gamma = g + G_\gamma$, i.e.\ $v(g-h) > \gamma$. Again, by maximality of $\gamma$, we obtain $g = h$. Hence, there is no element in $G$ strictly between $0$ and $g$, showing that $G$ is not densely ordered. This shows that (\ref{prop:denseeq:1}) implies (\ref{prop:denseeq:4}). For the converse, suppose that (\ref{prop:denseeq:4}) holds. We will show that $0$ is a limit point of $G$ in $G$. Let $g \in G$ with $g > 0$. If $v(g)$ is not maximal, then let $h \in G$ with $h > 0$ and $v(h) > v(g)$. Then $0<h<g$, as required. Otherwise, $\gamma = v(g)$ is the maximum of $vG$, and by assumption, $B_\gamma$ is densely ordered. Thus, there is some $h \in G$ with $v(h) = \gamma$ such that $0 + G_\gamma < h + G_\gamma < g + G_\gamma$, whence $0 < h < g$.
	\end{proof}

	{ 	Note that a divisible ordered abelian group $G$ never admits a proper non-trivial $\Log$-definable convex subgroup. Indeed, since $G$ is o-minimal, any $\Log$-definable convex subset of $G$ is an interval with endpoints in $G \cup \{\pm \infty\}$, and thus the only $\Log$-definable convex subgroups of $G$ are $\{0\}$ and $G$. The following dichotomy shows a strengthening of this observation.}

	\begin{proposition}\thlabel{prop:subgroupdef}
		Let $G$ be a densely ordered abelian group. Then exactly one of the following holds:
		\begin{enumerate}[wide, labelwidth=!, labelindent=6pt]
			\item $G$ is dense in $\div{G}$.
			
			\item $G$ has a proper non-trivial $\Log$-definable convex subgroup.
		\end{enumerate}
		Moreover, if $G$ is not dense in $\div{G}$, then the proper non-trivial $\Log$-definable convex subgroup of $G$ can be defined by an $\Log$-formula using one parameter.
	\end{proposition}
	
	\begin{proof}
		First assume that $G$ is not dense in $\div{G}$. In the following, we will construct a proper non-trivial convex subgroup of $G$.
		
		Note that by \Autoref{rmk:denseindivregular}, $G$ is non-archimedean. Thus, $vG$ has more than one element.
		Let $g_0 \in \div{G} \setminus \cl(G)$ with $g_0>0$. If $\gamma = v(g_0)$ is the maximum of $vG$, then let $h \in G^{ >0}$ with $v(h) < \gamma$. Note that $v(g_0 + h) = v(h) < \gamma$ and that $g_0 + h$ is not a limit point of $G$ in $\div{G}$, as $g_0$ is not a limit point of $G$ in $\div{G}$. Hence, $g_0 + h \in \div{G} \setminus \cl(G)$. Hence, by replacing $g_0$ by $g_0 + h$ if necessary, we may assume that $v(g_0)$ is not the maximum of $vG$.
		
		Let $g_1 \in G$ with $g_1>0$ and $N\in \N$ such that $g_0 = \frac{g_1}N$. Consider the set 
		$$D = \setbr{g \in G^{\geq 0} \mid g < g_0} = \setbr{g \in G^{\geq 0} \mid Ng < g_1}.$$ This set is $\Log$-definable with the parameter $g_1$. Let $$A = \setbr{g \in G^{\geq 0} \mid g + D \subseteq D}.$$
		Again, $A$ is $\Log$-definable with the parameter $g_1$. Note that $A$ is convex. Obviously, $A \neq G^{\geq0}$, as for any $g \in D$ and $h \in G^{\geq0}$ with $h>g_0-g$ we have $g + h > g_0$ and thus $h \notin A$. Assume that $A = 0$. Then for any $g \in G^{> 0}$, let $g' \in G^{\geq 0}$ with $g' < g_0$ and $g + g' > g_0$ and set $c_g = g + g'$. Then $g_0 < c_g < g_0 + g$.
		Since by \Autoref{prop:denseeq}~(\ref{prop:denseeq:3}), $0$ is a limit point of $G$ in $\div{G}$, for any $h \in \div{G}$ with $h > 0$, there exists some $g \in G$ such that  $0<g<h$. Thus, $g_0< c_g < g_0 + g < g_0 + h$. This shows that $g_0$ is a limit point of $G$ in $\div{G}$, contradicting the choice of $g_0 \notin \cl(G)$. Hence, $ A \neq 0$.
		
		Now for any $a,b \in A$ with $0<a<b$, we have $(a + b) + D \subseteq a + D \subseteq D$, whence $a+b \in A$. Moreover, $0<b-a<b$ and $-b<a-b<0$. Thus, by convexity, $b-a \in A$ and $a-b \in -A$. Similarly, for any $a,b \in -A$ with $a<b<0$, we have $a\pm b \in -A$ and $b-a \in A$. This shows that $H = -A \cup A$ is closed under addition and thus an $\Log$-definable convex subgroup of $G$. Since $0 \neq A \neq G^{\geq 0}$, we also have that $H$ is a proper non-trivial subgroup of $G$.
		
		Assume for the converse that $G$ is dense in $\div{G}$. By \Autoref{prop:regimpliesdense}, $G$ is regular. Thus, by \Autoref{fact:regrobinson}~(\ref{fact:regrobinson:1}) $G$ has an archimedean model $G'$, say. Let $\varphi(x,\ul{y})$ be an $\Log$-formula and let $\psi$ be the $\Log$-sentence
		stating that for any $\ul{y}$, if $\varphi(x,\ul{y})$ defines a proper convex subgroup, then it is the trivial one. Since $G'$ is archimedean, its only convex subgroups are $\{0\}$ and $G'$. Thus, $\psi$ holds in $G'$. By elementary equivalence, $\psi$ also holds in $G$. Hence, for any $\ul{b} \in G$, the formula $\varphi(x,\ul{b})$ does not define a proper non-trivial convex subgroup of $G$.
	\end{proof}

	Next we want to study the extension of ordered abelian groups $G \subseteq \div{G}$ valuation theoretically. 
	It is known that any dense extension of ordered fields is immediate (cf.\ \cite[page
	29~f.]{kuhlmann}). We will investigate when this is also the case for dense extensions of ordered abelian groups. Note that for any $\gamma \in vG$, we have $B(\gamma,\div{G}) = \div{B(\gamma,G)}$. Hence, the extension $G \subseteq \div{G}$ is immediate if and only if all archimedean components of $G$ are divisible.
	
	\begin{proposition}\thlabel{prop:densimpimm}
		Let $G\subseteq H$ be a dense extension of ordered abelian groups. Suppose that $vG$ has no last element. Then the extension $G \subseteq H$ is immediate.
	\end{proposition}
	
	\begin{proof}
		In order to show that $G \subseteq H$ is immediate, we need to show that for any $a \in H{ \setminus\{0\}}$ there exists $b \in G$ such that $v(a-b)> v(a)$. Note that the density of $G$ in $H$ implies $vG = vH$. Let $a \in H$ with $a>0$. Since $vG$ has no last element, there is some $c \in G^{>0}$ such that $v(c) > v(a)$. By density of $G$ in $H$, there is some $b \in G$ such that $a-c< b < a+c$. We obtain $v(a-b) \geq  v(c) > v(a)$, as required.
	\end{proof}
	
	\begin{corollary}\thlabel{cor:densimpimm}
		Let $G \subseteq \div{G}$ be a dense extension of ordered abelian groups such that $vG$ has no last element. Then the extension $G \subseteq \div{G}$ is immediate.
	\end{corollary}
	
	\begin{proof}
		Apply \Autoref{prop:densimpimm} to $H = \div{G}$.
	\end{proof}
	
	\Autoref{cor:densimpimm} does not hold in general in the case where $vG$ has a last element, as the following example will show.
	
	\begin{example}\thlabel{ex:nonarchdense}
		Let	$A = \setbr{\left.\frac{a}{2^n} \ \right| { a \in \Z, n\in \N_0}}$.
		Note that $A$ is dense in $\div{A} = \Q$. For $G = \Q \oplus A$ we have $\div{G} = \Q \oplus \Q$. Moreover, $G$ is dense in $\div{G}$, as $A$ is dense in $\Q$. However, the extension is not immediate, as the archimedean components do not coincide.
	\end{example}

	We  now consider the converse direction, that is, under what conditions an immediate extension of ordered abelian groups is dense.
	
	\begin{lemma}\thlabel{lem:denseimm}
		Let $G,A_1,A_2,\ldots$ be { non-trivial} ordered abelian groups such that $G\subseteq H = \H_{n \in \omega} A_n$ is immediate. Let $a \in \H_{n \in \omega} A_n$. Then there exists a sequence $(d_n)_{n \in \omega}$ in $G$ such that for any $k \in \omega$ and any $i \leq k$ we have $(d_0+\ldots+d_k)(i) = a(i)$.
	\end{lemma}
	
	\begin{proof}
		{ If $a\in G$, then set $d_0=a$ and $d_j=0$ for $j\geq 1$.
		Otherwise, since} $G \subseteq H$ is immediate, there is some $d_0 \in G$ such that $v(a-d_0) > v(a)$. Then $d_0(0) = a_0$. Suppose that $d_0,\ldots,d_k \in G$ are already constructed such that $d' = d_0+\ldots+d_k$ satisfies $d'(i) = a(i)$ for $i\leq k$. Again, since $G \subseteq H$ is immediate, there is some $d_{k+1} \in G$ such that $v((a-d')-d_{k+1}) > v(a-d') \geq k+1$. Thus, $d_{k+1}(i) = (a-d')(i) = 0$ for $i\leq k$ and $d_{k+1}(k+1) = (a-d')(k+1)$. We obtain $(d_0+\ldots + d_{k+1})(i) = a(i)$ for $i\leq k$ and $(d_0+\ldots + d_{k+1})(k+1) = (d' + d_{k+1})(k+1) = a(k+1)$, as required.
	\end{proof}
	
	\begin{proposition}\thlabel{prop:denseimm}
		Let $G \subseteq H$ be an immediate extension of ordered abelian groups such that $H$ is densely ordered and $vH \subseteq \omega$. Then $G$ is dense in $H$.
	\end{proposition}
	
	\begin{proof}
		Let $\Gamma = vH$. By the Hahn Embedding Theorem (cf.\ \cite[page~14]{kuhlmann}), we may consider $H$ as a subgroup of the Hahn product $\H_{n \in \omega} B_n$, where we set $B_n = \{0\}$ for $n \notin \Gamma$. Note that $H \subseteq  \H_{n \in \omega} B_n$ is an immediate extension.
		Let $a,b \in H$ with $0< a < b$. 
		By \Autoref{lem:denseimm}, there exists a sequence $(d_n)_{n \in \omega}$ in $G$ such that for any $k \in \omega$ and any $i \leq k$ we have $(d_0+\ldots+d_k)(i) = a(i)$. Suppose that { $k=v(b-a)$} is the last element of $vG$. Then by \Autoref{prop:denseeq}~(\ref{prop:denseeq:4}), $B_k$ is densely ordered. Let $c \in G$ with $v(c) = k$ and $a(k) < c(k) < b(k)$. Then $d' = d_1+\ldots+d_{k-1} + c \in G$ and $a < d' < b$, as required. { Now suppose that $k$ is not the last element of $vG$.} Then let $c \in G$ with $v(c) = m > k$ and $c(m) > a(m)$. Then $d' = d_1+\ldots+d_{m-1} + c \in G$ and $a < d' < b$, as required. 
	\end{proof}
	
	\begin{corollary}\thlabel{cor:denseimm}
		Let $G$ be an ordered abelian group. Suppose that the extension $G \subseteq \div{G}$ is immediate and that $vG \subseteq \omega$. Then $G$ is dense in $\div{G}$.
	\end{corollary}
	
	\begin{proof}
		Apply \Autoref{prop:denseimm} to $H = \div{G}$.
	\end{proof}
	
	Note that without the condition $vG \subseteq \omega$ in \Autoref{cor:denseimm}, the conclusion that $G$ is dense in $\div{G}$ does not hold in general, as the following example shows.
	
	\begin{example}
		Let $H$ be the Hahn product ${\H}_{\gamma \in \omega + 1} \Q$. Let $H'$ be the Hahn sum $H' = \coprod_{\gamma \in \omega + 1} \Q \subseteq H$. Note that $H'\subseteq H$ is an immediate extension (cf.\ \cite[page~3]{kuhlmann}). It follows that for any ordered abelian groups $G_1$ and $G_2$ with $H'\subseteq G_1 \subseteq G_2 \subseteq H$, also the extension $G_1\subseteq G_2$ is immediate.
		
		Let $G\subseteq H$ be given by 
		$$
		G = H' + a\Z \text{, where }a = \sum_{\gamma \in \omega} \mathds{1}_\gamma.$$
		Now $\div{G} = H' + a\Q  \subseteq H$, and the extension $G \subseteq \div{G}$ is immediate. Let $c = \frac 1 2 a + \frac 1 3 \mathds{1}_\omega$ and $d = \frac 1 2 a + \frac 2 3 \mathds{1}_\omega$. Then $c,d \in \div{G}$ with $0<c<d$. However, there is no element in $G$ strictly between $c$ and $d$. Thus $G$ is not dense in $\div{G}$.
	\end{example}
	
	\begin{remark}\thlabel{rmk:denseimmgps}
		By \Autoref{cor:densimpimm} and \Autoref{cor:denseimm}, we obtain the following: Let $G$ be an ordered abelian group and $vG \cong \omega$. Then the extension $G\subseteq \div{G}$ is dense if and only if it is immediate.
	\end{remark}
	
	The final results of this section are on ordered abelian groups which are not dense in their divisible hull. These will be used in \Autoref{sec:lrdef} to compare $\Lr$- and $\Lor$-definability of henselian valuations with real closed residue field in almost real closed fields. We start by giving some useful examples of ordered abelian groups which are not regular and thus not dense in their divisible hull.
		
	\begin{example}\thlabel{ex:regnew}
		\begin{enumerate}[wide, labelwidth=!, labelindent=6pt]
			\item  \thlabel{ex:regnew:1}
			$\Z \oplus \Z$ is a discretely ordered group which is not regular.
			
			\item  \thlabel{ex:regnew:2} Let $A$ be as in \Autoref{ex:nonarchdense}. Consider $G = A \oplus A$. Since $A$ is densely ordered, so is $G$ by \Autoref{prop:denseeq}. Since there is no element in $G$ between $\left(\frac 1 3 , 0\right)$ and $\left(\frac 1 3, 1\right)$, which both lie in $\div{G} = \Q \oplus \Q$, it follows that $G$ is not dense in $\div{G}$ and thus not regular. However, $G$ has limit points in $\div{G}\setminus G$, e.g.\ $\left(0,\frac 1 3\right)$.
			
		\end{enumerate}
	\end{example} 
	
	\Autoref{ex:regnew}~(\ref{ex:regnew:2}) shows that it is possible for an ordered abelian group to have limit points in its divisible hull but not to be dense in it. We can go even further and construct a non-divisible densely ordered abelian group $G$ which does not have any limit points in $\div{G} \setminus G$; or, in other words, $G$ is closed in $\div{G}$. We conclude this section by proving results on non-divisible densely ordered abelian groups which are closed in their divisible hull.
	A natural example of these are non-divisible Hahn products.
	
	\begin{proposition}\thlabel{prop:hahnnoarch}
		Let $\Gamma \neq \emptyset$ be a totally ordered set without a last element, and for any $\gamma \in \Gamma$, let $A_\gamma \neq \{0\}$ be an archimedean ordered abelian group. Suppose that $G=\H_{\gamma \in \Gamma} A_\gamma$ is non-divisible. Then $G$ is closed in $\div{G}$.
	\end{proposition}
	
	\begin{proof}
		Let $s \in \div{G}\setminus G$, and let $\gamma_0 \in \Gamma$ with $\gamma_0 > \min\{ \gamma \in  \Gamma \mid s_\gamma \notin A_\gamma\}$. Such $\gamma_0$ exists, as $\Gamma$ has no last element. Then the interval $(s-\mathds{1}_{\gamma_0}, s+\mathds{1}_{\gamma_0})$ does not contain any element of $G$. Hence, $s$ is not a limit point of $G$.
	\end{proof}

	\Autoref{prop:hahnnoarch} shows in particular that any non-divisible Hahn product whose value set has no last element is not dense in its divisible hull. Thus, by \Autoref{fact:regrobinson}~(\ref{fact:regrobinson:1}), any such Hahn product has no archimedean model. A similar result holds for ordered Hahn fields which are not real closed (see \Autoref{rmk:hahnfieldnoarch}).
	The final results of this section gives us sufficient conditions for an ordered abelian group to be closed in its divisible hull.
	
	\begin{proposition}\thlabel{prop:convdivlimit}
		Let $G$ be an ordered abelian group. Suppose that for every prime $p \in \N$, there is a $p$-divisible convex subgroup $G_p\neq \{0\}$ of $G$. Then $G$ is closed in $\div{G}$.
	\end{proposition}
	
	\begin{proof}
		If $G$ is divisible, then the conclusion is trivial. Otherwise, let $a \in G$ and $N \in \N$ such that $\frac a N \in \div{G}\setminus G$. Let $N = p_1\ldots p_m$ be the prime factorisation of $N$. Consider the non-trivial $N$-divisible convex subgroup $H = \bigcap_{i=1}^m G_{p_i}$. Let $h \in H$ with $h>0$. Consider the interval $I=\brackets{\frac aN - h, \frac aN + h}$ in $\div{G}$. Assume that $ I \cap G \neq \emptyset$. Then for any $g \in I \cap G$ we have $v(Ng-a) = v\brackets{g-\frac aN} \geq v(h)$ and thus $Ng- a \in H$. Since $H$ is $N$-divisible, we obtain $g - \frac aN \in H$ and thus $\frac aN\in G$, a contradiction. This gives us  $I \cap G = \emptyset$, as required.
	\end{proof}
	
	\begin{corollary}\thlabel{cor:convdivlimit}
		Let $G$ be an ordered abelian group. Suppose that $G$ has a convex divisible subgroup $H \neq \{0\}$. Then $G$ is closed in $\div{G}$.
	\end{corollary}
	
	\begin{proof}
		For any prime $p$, let $G_p = H$. The conclusion follows from \Autoref{prop:convdivlimit}.
	\end{proof}
	
	The following example shows that the converse of \Autoref{prop:convdivlimit} does not hold.
	
	\begin{example}\thlabel{ex:conversedivclosed}
		Let $2=p_0<p_1<p_2<\ldots$ be a complete list of all prime numbers. For $n \in \N$, let $A_n$ be the following ordered abelian group
		$$A_n = \setbr{\left. \frac{a}{p_1^{m_1}\ldots p_n^{m_n}}  \ \right| { a\in \Z,m_1,\ldots,m_n \in \N_0 }}.$$
		Then $A_n$ is $p_i$-divisible for $i=1,\ldots,n$. 
		Let $G = \H_{n \in \N} A_n$. Note that $\div{G} = \H_{n \in \N} \Q$. Since $A_n$ is not $2$-divisible for any $n \in \N$, there is no non-trivial $2$-divisible convex subgroup of $G$. In particular $G$ is not divisible and thus closed in $\div{G}$ by 		
		\Autoref{prop:hahnnoarch}. Moreover, for any prime $p_i \neq 2$, the maximal $p_i$-divisible subgroup of $G$ is $ \H_{n \in \N_{\geq i}} A_n$.

	\end{example}

	\section{Density in Real Closure}\label{sec:density}
	
	We now turn to ordered fields which are dense in their real closure.
	
	\subsection{Model theoretic properties}\label{sec:modeltheorydensity}
	In this subsection, we study model theoretically the class of ordered fields which are dense in their real closure.
	At first, we change to a more general setting of complete theories which admit quantifier elimination. 
	For a structure $\mathcal{M}$ and a subset $A \subseteq M$, we denote the definable closure of $A$ in $\mathcal{M}$ by $\dcl(A;\mathcal{M})$.
	
	\begin{theorem}\thlabel{prop:omindenserec}
		Let $\L$ be a language expanding $\Log$, let $T\supseteq T_{\mathrm{doag}}$ be a complete $\L$-theory which admits quantifier elimination and let $\Sigma \subseteq T$ be a theory extending the theory of ordered abelian groups.
		Then there is a theory $\Sigma' \supseteq \Sigma$ such that for any $\MM' \models T$ and any $\MM \subseteq \MM'$ with $\dcl(M;\MM') = M'$ we have that $\MM \models \Sigma'$ if and only if $\MM \models \Sigma$ and $M$ is dense in $M'$. Moreover, if $\Sigma$ and $T$ are recursive, then we can also choose $\Sigma'$ to be recursive.
	\end{theorem}
	
	\begin{proof}{ 
		For any $0$-definable function $f$ in $T$, 
		let $\varphi_f(\ul{x},z)$ be a quantifier-free formula which is equivalent to $z < f(\ul{x})$ in $T$. Let $A$ be the set of all pairs $(f,g)$ of $0$-definable functions in $T$ with the same arity such that $$T \models \forall \ul{x} \ (f(\ul{x}) = g(\ul{x}) \to \ul{x} = 0).$$ In other words, $A$ consists of all pairs of $0$-definable functions which are distinct everywhere except possibly in $0$.
		Set 
		$$\Sigma' = \Sigma \cup \setbr{\forall \ul{x}\ (\ul{x} \neq 0 \to \neg \forall z\  (\varphi_f(\ul{x},z) \leftrightarrow \varphi_g(\ul{x},z) )) \mid (f,g) \in A}.$$ Note that if $\Sigma$ and $T$ are recursive, then so is $\Sigma'$, as it is decidable whether a given $\L$-formula defines a function. 
		
		Let $\MM' \models T$ and let $\MM \subseteq \MM'$ with $\dcl(M;\MM') = M'$. Suppose that $\MM \models \Sigma'$. We need to show that $M$ is dense in $M'$. Let $\alpha, \beta \in M'$ such that $\alpha \neq  \beta$. Let $f$ and $g$ be $0$-definable functions and let $\ul{a} \in M$ such that $f(\ul{a}) = \alpha$ and $g(\ul{a}) = \beta$. Since also $\dcl(M\setminus\setbr{0};\MM') = M'$, we may assume that $\ul{a} \neq 0$. Moreover, we may assume that $(f,g) \in A$, as otherwise we can replace $g$ by the $0$-definable function $$\ul{x}\mapsto\begin{cases}g(\ul{x})&\text{ if } f(\ul{x}) \neq g(\ul{x}) \text{ or } \ul{x}=0,\\|\ul{x}|+|f(\ul{x})|&\text{ otherwise.}\end{cases}$$ Hence, $$\MM \models \neg \forall z\  (\varphi_f(\ul{a},z) \leftrightarrow \varphi_g(\ul{a},z) ).$$ This implies that for some $b \in M$ we have 
		$$\MM \models \varphi_f(\ul{a},b) \wedge \neg \varphi_g(\ul{a},b)\text{ or }\MM \models \neg \varphi_f(\ul{a},b) \wedge \varphi_g(\ul{a},b).$$ Since $\varphi_f$ and $\varphi_g$ are quantifier-free, we obtain that either $$\MM' \models \varphi_f(\ul{a},b) \wedge \neg \varphi_g(\ul{a},b)\text{ or }\MM' \models \neg \varphi_f(\ul{a},b) \wedge \varphi_g(\ul{a},b).$$ Hence, either $$\beta = g(\ul{a}) \leq b < f(\ul{a}) = \alpha\text{ or }\alpha = f(\ul{a}) \leq b < g(\ul{a}) = \beta.$$ This implies that $M$ is dense in $M'$.
		
		Conversely, suppose that $\MM \models \Sigma$ and $M$ is dense in $M'$. Let $(f,g) \in A$ and $\ul{a} \in M$ with $\ul{a} \neq 0$. We need to show that $$\MM \models \neg \forall z\  (\varphi_f(\ul{a},z) \leftrightarrow \varphi_g(\ul{a},z) ).$$ Let $\alpha = f(\ul{a})$ and $\beta = g(\ul{a})$. By assumption on $A$, we have $\alpha \neq \beta$, say $\alpha < \beta$. Since $M$ is dense in $M'$, there is $c \in M$ such that $\alpha < c < \beta$. Hence, $\MM' \models \neg \varphi_f(\ul{a},c)$ but $\MM' \models \varphi_g(\ul{a},c)$. Since $\varphi_f$ and $\varphi_g$ are quantifier-free, we obtain $\MM \models \neg \varphi_f(\ul{a},c)$ and  $\MM \models \varphi_g(\ul{a},c)$, as required.}
	\end{proof}

	\begin{corollary}\thlabel{lem:densedc}
		Let $\L$ be a language expanding $\Log$ and let $\MM=(M,+,0, \linebreak  <,\ldots)$ and $\NN=(N,+,0,<,\ldots)$ be ordered $\L$-structures such that $(M,+,0,  <)$ is a non-trivial ordered abelian group and $\MM \equiv \NN$. Suppose that there exists a complete $\L$-theory $T\supseteq T_{\mathrm{doag}}$ admitting quantifier elimination such that there are $\MM',\NN' \models T$ with $\MM\subseteq \MM'$, $\NN\subseteq \NN'$, $\dcl(M;\MM')=M'$ and $\dcl(N;\NN')=N'$. Suppose further that $M$ is dense $M'$. Then $N$ is also dense in $N'$.
	\end{corollary}
	
	\begin{proof}
		By applying \Autoref{prop:omindenserec} to $\Sigma = \Th(\MM)$, we obtain a theory $\Sigma'$ such that $\MM \models \Sigma'$. Moreover, since $\MM \equiv \NN$, we obtain $\NN\models \Sigma'$. Hence, $N$ is dense in $N'$.
	\end{proof}

	\Autoref{prop:omindenserec} is applicable to both the $\Log$-theory of divisible ordered abelian groups $\Tdoag$ and the $\Lor$-theory of real closed fields $\Trcf$. Note that by o-minimality of $T_{\mathrm{doag}}$, for any extension of ordered abelian groups $G \subseteq H$ where $H$ is divisible, we have $\dcl(G;H) = \div{G}$. Similarly, for any extension of ordered fields $(K,<) \subseteq (F,<)$ such that $F$ is real closed, we have $\dcl(K;(F,<))=\rc{K}$.
	
	\begin{corollary}\thlabel{cor:denserec}
		\begin{enumerate}[wide, labelwidth=!, labelindent=6pt]
			\item  		There is a recursive $\Log$-theory which axiomatises the class of non-trivial ordered abelian groups which are dense in their divisible hull.
			
			\item\label{cor:denserec:2}  		There is a recursive $\Lor$-theory which axiomatises the class of ordered fields which are dense in their real closure.
		\end{enumerate}
	\end{corollary}
	
	\begin{proof}
		\begin{enumerate}[wide, labelwidth=!, labelindent=6pt]
			\item\label{cor:denserec:pf:1}  	Let $\Sigma$ be the recursive $\Log$-theory of non-trivial ordered abelian groups. By applying \Autoref{prop:omindenserec} to $T = T_{\mathrm{doag}}$, we obtain a recursive $\Log$-theory $\Sigma'$ such that for any non-trivial divisible ordered abelian group $H$ and for any subgroup $G \subseteq H$ with $\div{G} = \dcl(G;H) = H$ we have that $G \models \Sigma'$ if and only if $G$ is dense in $H = \div{G}$. Hence, $\Sigma'$ axiomatises the class of non-trivial ordered abelian groups which are dense in their divisible hull.
			
			\item  	We can apply \Autoref{prop:omindenserec} to $T = \Trcf$ and then argue as in (\ref{cor:denserec:pf:1}) to obtain a recursive axiomatisation of the class of ordered fields which are dense in their real closure.
		\end{enumerate}
	\end{proof}
	
	\begin{remark}
		We have already shown in \Autoref{cor:regaxiom} that there exists a recursive axiomatisation of non-trivial ordered abelian groups which are dense in their divisible hull and stated the axiomatisation explicitly. An explicit axiomatisation of ordered fields which are dense in their real closure is given in \cite[Satz~13]{hauschild} as wells as \cite[Theorem~5]{mckenna} as the theory of ordered fields $T_{\mathrm{of}}$ plus the following axioms (one for each $n \in \N$):
		\begin{align*}&\forall c_0,\ldots,c_n,a,b,\varepsilon\\
		&[(a<b\wedge 0 < \varepsilon \wedge F(a)F(b)<0) \to \exists z \ (a<z<b \wedge |F(z)|<\varepsilon)],
		\end{align*}
		 where $F(x)=c_nx^n + \ldots + c_0$.
		We point out that \Autoref{prop:omindenserec} is applicable in a wider context than ordered abelian groups and ordered fields, as its proof gives a general procedure to obtain recursive axiomatiations of ordered structures which are dense in certain definable closures.
	\end{remark}
	
	We now apply the results above to the o-minimal complete theory of real closed fields $\Trcf$.
	
	\begin{corollary}	\thlabel{thm:densitytransfers}	
		Let $(K,<)$ and $(L,<)$ be ordered fields such that $(K,<) \equiv (L,<)$. Suppose that $K$ is dense in $\rc{K}$. Then $L$ is dense in $\rc{L}$.
	\end{corollary}
	
	\begin{proof}
		This follows immediately from \Autoref{cor:denserec}~(\ref{cor:denserec:2}).
	\end{proof}

	Let $(K,<)$ be an archimedean ordered field. Then $\Q \subseteq K \subseteq \rc{K} \subseteq \R$. Since $\Q$ is dense in $\R$, also $K$ is dense in $\rc{K}$. In other words, any archimedean ordered field is dense in its real closure. We thus obtain the following.
	
	\begin{corollary}\thlabel{thm:densitytransfers2}	
		Let $(K,<)$ be an ordered field which has an ar\-chi\-me\-de\-an model. Then $K$ is dense in $\rc{K}$.
	\end{corollary}
	
	\begin{proof}
		Since any archimedean ordered field is dense in its real closure, we obtain from \Autoref{thm:densitytransfers} that $K$ is dense in $\rc{K}$.
	\end{proof}

	By \Autoref{thm:densitytransfers2}, any ordered field with an archimedean model is dense in its real closure. We will use \Autoref{cor:denserec} to show that the converse does not hold, i.e.\ that there exists an ordered field which is dense in its real closure but has no archimedean model.\footnote{ In this regard, we thank Erik Walsberg for pointing out a related question.}
	Let $(K,<)$ be an ordered field. An \textbf{integer part} of $K$ is a discretely ordered subring $(Z,<) \subseteq (K,<)$ with $1$ as least positive element such that for any $x \in K$ there exists $z\in Z$ with $z \leq x < z+1$. The connections between integer parts and density in real closure are further explored in \Autoref{sec:ip}.\label{page:ip}
	
	\begin{proposition}\thlabel{prop:ordfieldnoarchmodel}
		There is an ordered field $(K,<)$ such that $K$ is dense in $\rc{K}$ but $(K,<)$ has no archimedean model. Moreover, $(K,<)$ does not admit a non-trivial $\Lor$-definable convex valuation.
	\end{proposition}
	
	\begin{proof}
		Let $\varphi(x)$ be an $\Lor$-formula defining $\Z$ in $(\Q,<)$ (cf.\ \cite[The\-o\-rem~3.1]{robinson2}). Let $\Sigma_1$ be the $\Lor$-theory $\Sigma'$ from \Autoref{cor:denserec} and let $\Sigma_2$ be a set of axioms stating that $\varphi(x)$ defines an integer part. Moreover, let $\Sigma_3$ consist of all $\Lor$-sentences stating that $\psi(x)$ does not define a non-trivial convex valuation ring, where $\psi(x)$ ranges over all $\Lor$-formulas with one free variable. For any $\Lor$-formula $\alpha$, let $\widetilde{\alpha}$ be the formula in which all quantifiers $\exists x$ and $\forall x$ are bounded by $\varphi(x)$, that is, all instances of subformulas of the form $\exists x \theta(x,\ul{y})$ are replaced by $\exists x (\varphi(x) \wedge \theta(x,\ul{y}))$ and all instances of subformulas of the form $\forall x \theta(x,\ul{y})$ are replaced by $\forall x (\varphi(x) \to \theta(x,\ul{y}))$. 
		
		Let $\Sigma$ be the deductive closure of $\Sigma_1 \cup \Sigma_2\cup \Sigma_3$. Note that $\Sigma$ is recursive  and $(\Q,<) \models \Sigma$.
		If for every $\Lor$-sentence $\alpha \in \Th(\Z,+,-,\cdot,0,1,<)$, we had $\Sigma \vdash \widetilde{\alpha}$, then $\Th(\Z,+,-,\cdot,0,1,<)$ would be decidable. Hence, there is an $\Lor$-sentence $\sigma$ such that $\neg \sigma \in \Th(\Z,+,-,\cdot,0,1,<)$ and $\widetilde{\sigma},\neg\widetilde{\sigma} \notin \Sigma$. Let $\widetilde{\Sigma}$ be the consistent $\Lor$-theory $\Sigma \cup \setbr{\widetilde{\sigma}}$ and let $(K,<) \models \widetilde{\Sigma}$. Assume that $(K,<)$ is archimedean. Since $(K,<) \models \Sigma_2$, the formula $\varphi(x)$ defines the integer part $\Z$ of $K$. But then $(K,<) \models \neg \widetilde{\sigma}$, a contradiction. Hence, $(K,<)$ cannot have any archimedean models. However, since $(K,<) \models \Sigma_1$, we have that $K$ is dense in $\rc{K}$. Moreover, since $(K,<)\models \Sigma_3$, it does not admit a non-trivial $\Lor$-definable convex valuation.
	\end{proof}

	\subsection{Dense and immediate extensions}\label{sec:denseimm}

	We now analyse algebraic properties of ordered fields which are dense in their real closure. As pointed out 
	above, there is an intimate connection between integer parts and dense extensions of ordered fields. This connection is explored in \cite{biljakovic} and will further be addressed here. 
	Throughout this section, we denote the natural valuation $\vnat$ simply by $v$. 

	We start by investigating the connection between dense and immediate extensions of ordered fields.
	Let $(K,<) \subseteq (L,<)$ be an extension of ordered fields and let $w$ be a convex valuation on $L$. We say that $K$ is \textbf{$w$-dense} in $L$ if for any $a \in L$ and any $\alpha \in wL$, there exists $b \in K$ such that $w(a-b) > \alpha$. The following result is due to the fact that the topology induced by a non-trivial convex valuation on an ordered field coincides with the order topology (cf.\ e.g.\ \cite[Section~7.63]{alling}).
	
	\begin{fact}\thlabel{prop:wdense}
		Let $(K,<) \subseteq (L,<)$ be an extension of non-archimedean ordered fields and let $w$ be a non-trivial convex valuation on $L$. Then $K$ is $w$-dense in $L$ if and only if $K$ is dense in $L$.
	\end{fact}
	


	By \cite[Lemma~1.31]{kuhlmann}, any $w$-dense extension of ordered fields is also immediate with respect to $w$. We thus obtain the following corollary to \Autoref{prop:wdense}.
	
	\begin{corollary}\thlabel{cor:denseimmfields}
		Let $(K,<) \subseteq (L,<)$ be a dense extension of non-ar\-chi\-me\-de\-an ordered fields. Then this extension is immediate with respect to any non-trivial convex valuation on $L$.
	\end{corollary}
	
	\Autoref{cor:denseimmfields} can be applied to ordered Hahn fields as follows: 
	Let $(k,<)$ be an ordered field and $G\neq \{0\}$ an ordered abelian group. Suppose that $k\pow{G}$ is dense in $\rc{k\pow{G}}$. By \Autoref{cor:denseimmfields}, $k\pow{G} \subseteq \rc{k\pow{G}}$ is an immediate extension with respect to the valuation $\vmin$. Hence, $k= \rc{k}$ and $G = \div{G}$. This implies that $k\pow{G}$ is real closed. Hence, if $k\pow{G}$ is not real closed, then it is not dense in its real closure. This can be strengthened with the following proposition.
	
	\begin{proposition}\thlabel{prop:hahnclosed}
		Let $(k,<)$ be an ordered field and $G\neq \{0\}$ an ordered abelian group. Then ${k\pow{G}}$ is closed in $\rc{k\pow{G}}$.
	\end{proposition}
	
	\begin{proof}
		Let $K = k\pow{G}$ and $ R = \rc{K}$. We need to show that $R\setminus K$ is open in $R$. If $K=R$, then $R\setminus K = \emptyset$. Hence, assume that $K \subsetneq R$. Let $at^{g_0}$ be the monomial of $s$ of least exponent which is not contained in $K$, and let $g_1 \in G^{>g_0}$ be arbitrary. Then the open interval
		$$I=(s-t^{g_1},s+t^{g_1}) \subseteq R$$
		contains $s$. However, any element in $I$ contains the monomial $at^{g_0}$ and is thus not contained in $K$. Hence, $s$ is contained in an open neighbourhood in $R \setminus K$, as required.
	\end{proof}

	\begin{corollary}\thlabel{rmk:hahnfieldnoarch}
		Let $(k,<)$ be an ordered field and let $G \neq \{0\}$ be an ordered abelian group. Suppose that $k\pow{G}$ is not real closed. Then $(k\pow{G},<)$ has no archimedean model.
	\end{corollary}
	
	\begin{proof}
		By \Autoref{prop:hahnclosed}, $k\pow{G}$ is not dense in $\rc{k\pow{G}}$. Thus, by \Autoref{thm:densitytransfers2} $(k\pow{G},<)$ has no archimedean model.
	\end{proof}

	We obtain the following result regarding the existence of henselian valuations in ordered fields which are dense in their real closure.
	
	\begin{corollary}\thlabel{cor:nohensval}
		Let $(K,<)$ be an ordered field which is not real closed but dense in $\rc{K}$. Then $K$ does not admit a non-trivial henselian valuation.
	\end{corollary}

	\begin{proof}
		If $K$ is archimedean, then the only possible henselian valuation on $K$ is the trivial one. If $K$ is non-archimedean, let $w$ be a henselian valuation on $K$. By \Autoref{fact:ake}, we have $(K,<) \equiv (Kw\pow{wK},<)$. By \Autoref{thm:densitytransfers}, $Kw\pow{wK}$ is dense in $\rc{Kw\pow{wK}}$. By \Autoref{prop:hahnclosed}, this is only possible if $wK = \{0\}$, i.e.\ $w$ is the trivial valuation.
	\end{proof}
	
	We now focus on immediate extension with respect to the natural valuation $v$.
	For an ordered field $K$, let $\mathbf{A}_K$ be a group complement to $\OO_v$ in $K$, i.e.\ $\mathbf{A}_K$ is an ordered subgroup of $K$ such that $K = \mathbf{A}_K \oplus \OO_v$. Note that for an extension of ordered field $(K,<) \subseteq (L,<)$, the group complement $\mathbf{A}_K$ can be extended to a group completement $\mathbf{A}_L$. The following is a useful characterisation of dense extensions of ordered field (cf.\ \cite[Lemma~1.32]{kuhlmann}).
	
	\begin{fact}\thlabel{fact:complements}
		Let $(K,<) \subseteq (L,<)$ be an extension of ordered fields. Then $K$ is dense in $L$ if and only if $\mathbf{A}_K = \mathbf{A}_L$.
	\end{fact}
	
	Whenever the extension of ordered fields $(K,<)\subseteq (\rc{K},<)$ is dense and thus immediate, then $Kv$ is real closed and $vK$ is divisible. The converse holds under the additional assumption that $vK$ is archimedean (cf.\ \cite[Proposition~9]{viswanathan}).
	
	\begin{fact}\thlabel{fact:immimpdense}
		Let $(K,<)$ be an ordered field such that $Kv$ is real closed and $vK$ is divisible and archimedean. Then $K$ is dense in $\rc{K}$.
	\end{fact}
	
	\cite[Remark~3.5]{biljakovic}
	provides an example of an ordered field field $(K,<)$ such that $Kv$ is real closed and $vK$ can be chosen to be divisible and non-archimedean (i.e.\ $K\subseteq \rc{K}$ is immediate) but $\mathbf{A}_K \neq \mathbf{A}_{\rc{K}}$, whence $K$ is not dense in $\rc{K}$.
	
	\subsection{Normal integer parts}\label{sec:ip}
	
	We now turn to normal integer parts in connection to ordered fields dense in their real closure.
	Due to \cite{mourgues}, every real closed field admits an integer part. In \cite[Remark~3.3~(iii)]{biljakovic} the question is asked whether every real closed fields admits a \textbf{normal} integer part, i.e.\ an integer part $Z$ which is integrally closed in its field of fractions $\ff(Z)$.
	
	\begin{remark}\thlabel{rmk:normalip}
			 For an ordered field $(K,<)$, if $Z$ is an integer part of $K$, then $\ff(Z)$ is dense in $K$. Moreover, if $Z$ is normal, then $\Q\subseteq \ff(Z)$ is regular, i.e.\ $\Q$ is relatively algebraically closed in $\ff(Z)$. Hence, if an ordered field admits a normal integer part, then it has a dense subfield which is regular over $\Q$. Note that any archimedean ordered field has $\Z$ as a normal integer part. Hence, the following question is closely related to the one in \cite[Remark~3.3~(iii)]{biljakovic}: 
			Given a non-archimedean real closed field $(K,<)$, does it contain a dense subfield $F \subseteq K$ such that $F$ is regular over $\Q$? In the following, we will give a positive answer to this in the case that the absolute transcendence degree of $K$ is infinite.
	\end{remark}

	For an ordered field $(K,<)$, we denote its absolute transcendence degree, i.e.\ its transcendence degree over $\Q$,  by $\td(K)$. We say that $T \subseteq K$ is a transcendence basis of $K$ if it is a transcendence basis of $K$ over $\Q$.
	The following result is due to \cite[Lemma~2.3]{erdos}.
	
	\begin{fact}\thlabel{lem:lemmaerdos}
		Let $(K,<)$ be an uncountable ordered field. Then $K$ admits a dense transcendence basis.
	\end{fact}
	
	The arguments in \cite{erdos} can be generalised to countable ordered fields with countably infinite transcendence basis.
	
	\begin{proposition}\thlabel{prop:lemmaerdos}
		Let $(K,<)$ be a countable ordered field with $\td(K) = \aleph_0$. Then $K$ admits a dense transcendence basis.
	\end{proposition}
	
	\begin{proof}
		Let $(I_n)_{n\in\N}$ be an enumeration of all intervals $(a,b) \subseteq K$. We construct a transcendence basis $\{t_1,t_2,\ldots\}$ of $K$ over $\Q$ such that $t_k \in I_k$ for any $k$.
		Let $t_1 \in I_1$ be an arbitrary element transcendental over $\Q$. Suppose that $t_1,\ldots,t_{n}$ have already been chosen for some $n$. Assume, for a contradiciton, that all elements in $I_{n+1}$ were algebraic over $\Q(t_1,\ldots,t_n)$. Then by \cite[Lemma~2.2]{erdos}, also $K$ is algebraic over $\Q(t_1,\ldots,t_n)$. This contradicts $\td(K) = \aleph_0$. Hence, we can choose $t_{n+1} \in I_{n+1}$ which is transcendental over $\Q(t_1,\ldots,t_n)$. Now $T'=\{t_n \mid n \in \N\}$ is a dense subset of $K$ and algebraically independent over $\Q$ and can thus be extended to a dense transcendence basis $T \supseteq T'$ of $K$.
	\end{proof}
	
	\begin{corollary}\thlabel{cor:regularext}
		Let $(K,<)$ be an ordered field with $\td(K) \geq \aleph_0$. Then there is a dense subfield $F$ of $K$ such that $\Q \subseteq F$ is regular.
	\end{corollary}
	
	\begin{proof}
		By \Autoref{lem:lemmaerdos} and  \Autoref{prop:lemmaerdos}, $K$ has a dense transcendence basis $T$. Then $F=\Q(T)$ is dense in $K$ and $\Q \subseteq F$ is regular.
	\end{proof}

	Note that the conclusion of \Autoref{cor:regularext} also holds for ordered fields of transcendence degree $0$. Indeed, if $(K,<)$ is an ordered field with $\td(K) = 0$, then $K$ is algebraic over $\Q$ and thus archimedean. We can then set $F = \Q$. Thus, the question becomes whether the same conclusion can be made for non-archimedean ordered fields of non-zero finite transcendence degree (see \Autoref{qu:denseregular}).
	
	As a final result of this section, we obtain the following characterisation of ordered fields dense in their real closure in terms of dense transcendence bases.
	
	\begin{corollary}\thlabel{cor:lemmaerdos}
		Let $(K,<)$ be an ordered field. Suppose that $\td(K) \geq \aleph_0$. Then $K$ is dense in $\rc{K}$ if and only if $K$ has a transcendence basis $T$ which is dense in $\rc{K}$.
	\end{corollary}
	
	\begin{proof}
		Note that any transcendence basis of $K$ is a transcendence basis of $\rc{K}$. If $K$ has a transcendence basis $T$ which is dense in $\rc{K}$, then $K$ has a proper subset dense in $\rc{K}$ and thus $K$ itself is dense in $\rc{K}$. Conversely, suppose that $K$ is dense in $\rc{K}$. By \Autoref{lem:lemmaerdos} and  \Autoref{prop:lemmaerdos}, $K$ has a transcendence basis $T$ which is dense in $K$. Since $K$ is dense in $\rc{K}$, also $T$ is dense in $\rc{K}$.
	\end{proof}

	\section{Definable Convex Valuations}\label{sec:defconvval}
	
	In this section, we firstly investigate what convex valuations are $\Lor$-de\-fin\-a\-ble in ordered fields and secondly we compare our results to known $\Lr$-definability results of henselian valuations with special focus on almost real closed fields.

	\subsection{$\Lor$-definability}\label{subsec:def}
	
	We are going to analyse the construction method of $\Lor$-definable convex valuations from \cite[Proposition~6.5]{jahnke}.
	
	\begin{fact}\thlabel{prop:defval}
		Let $(K,<)$ be an ordered field. Then at least one of the following holds.
		\begin{enumerate}[wide, labelwidth=!, labelindent=6pt]
			\item $K$ is dense in its real closure.
			\item $K$ admits a non-trivial $\L_{\mathrm{or}}$-definable convex\footnote{\cite[Proposition~6.5]{jahnke} only states that $K$ admits a non-trivial $\Lor$-definable valuation, but the proof indeed gives a construction method for a non-trivial $\Lor$-definable \emph{convex} valuation.} valuation.
		\end{enumerate}
	\end{fact}
	
	Note that a real closed field $(K,<)$ does not admit any non-trivial $\Lor$-definable convex valuation, as by o-minimality, any $\Lor$-definable convex subset of $K$ must be an interval with endpoints in $K \cup \{\pm \infty\}$.
	We will summarise the construction procedure of a non-trivial $\L_{\mathrm{or}}$-definable convex valuation ring of an ordered field which is not dense in its real closure given in \cite[page~163~f.]{jahnke}. Recall that for an ordered field $(K,<)$, we denote its topological closure in $\rc{K}$ under the order topology by $\cl(K)$.
	
	\begin{construction}\thlabel{constr:val}
		Let $(K,<)$ be an ordered field. Suppose that $K$ is not dense in $R=\rc{K}$. Let $s\in R\setminus \cl(K)$. Set $D_s:=\{z\in K \mid z < s\}$ and $A_s := \{x \in K^{\geq0} \mid x+D_s\subseteq D_s\}$. Set $\OO_s := \{x \in K \mid |x|A_s \subseteq A_s\}$. Then $\OO_s$ is a non-trivial $\Lor$-definable convex valuation ring of $K$.
	\end{construction}
	

	
	\begin{theorem}\thlabel{lem:defvalnew}\thlabel{thm:defval}
		Let $(K,<)$ be an ordered field and $v$ a henselian valuation on $K$. Suppose that at least one of the following holds.
		
		\begin{enumerate}[wide, labelwidth=!, labelindent=6pt]
			\item\thlabel{lem:defvalnew:1} $vK$ is discretely ordered.
			
			\item\thlabel{lem:defvalnew:2} $vK$ has a limit point in $\div{vK}\setminus vK$.
			
			\item\thlabel{lem:defvalnew:3} $Kv$ has a limit point in $\rc{Kv}\setminus Kv$.
		\end{enumerate}
		Then $v$ is $\Lor$-definable in $K$. Moreover, in the cases (\ref{lem:defvalnew:1}) and (\ref{lem:defvalnew:2}), $v$ is definable by an $\Lor$-formula with one parameter.
	\end{theorem}
	
	\begin{proof}
		Let $k=Kv$ and $G=vK$. 
		Note that $K$ is not real closed, as in each case either $k$ is not real closed or $G$ is non-divisible. 
		By \Autoref{fact:ake}, $(K,\linebreak   <, v) \equiv (k\pow{G},<, \vmin)$. Let $L = k\pow{G}$, $w = \vmin$ and $R = \rc{L}$.
		We will apply \Autoref{constr:val} with some simplifications to define the valuation ring $k\pow{G^{\geq 0}}$ of $w$ in $L$. \Autoref{prop:hahnclosed} shows that we can apply the construction procedure to any element in $s \in  R \setminus L$.
		
		First suppose that $G$ is non-divisible. Let $g_0 \in \div{G}\setminus G$ and $s = t^{g_0}$.
		Consider the $\Lor$-definable set $D_s' = \setbr{x \in L^{\geq 0} \mid x < t^{g_0}}$. Since $g_0 \in \div{G}$, there is some $h \in G$ and $N \in \N$ such that $g_0 = \frac{h}{N}$. Thus, the set $D_s'$ is defined by the $\Lor$-formula with one parameter $$x \geq 0 \wedge x^N < t^h.$$ Note that for any $x \in L^{\geq0}$, we have $x \in D_s'$ if and only if $w(x) > g_0$. Thus, $D_s' = k\pow{G^{>g_0}}^{\geq0}$. 
		Let $\OO_s = \{x \in L \mid |x|D_s' \subseteq D_s'\}$. Note that this set is $\Lor$-definable with one parameter. By definition, $\OO_s$ contains exactly those elements in $L$ such that for any $y\in L^{\geq 0}$ with $w(y) > g_0$ we have \begin{align}w(x)+w(y) = w(xy) > g_0. \label{eq:cond1}\end{align} In particular, for any $x \in L$ with $w(x) \geq 0$, condition \eqref{eq:cond1} holds. Thus, $k\pow{G^{\geq0}} \subseteq \OO_s$. To show the other set inclusion, we will make a case distinction, also specifying the element $g_0$ for the densely ordered case.
		
		Suppose that $G$ is discretely ordered. Let $g_1 \in G$ be the least element greater than $g_0$ and let $g_2 \in G$ be the least element greater than $g_0-g_1$. Then $g_2+g_1$ is the least element greater than $g_0$. By choice of $g_1$, this gives us $g_2+g_1 = g_1$ and thus $g_2 = 0$.
		Let $x \in \OO_s$. Since $t^{g_1} \in D_s'$, we have $w(xt^{g_1}) = w(x) + g_1 > g_0$. Hence, $w(x) > g_0 - g_1$. By choice of $g_2$ as the least element greater than $g_0 - g_1$, we obtain $w(x) \geq g_2 = 0$. This implies $\OO_s \subseteq k\pow{G^{\geq0}}$, as required.
		
		Suppose that $G$ has a limit point in $\div{G}\setminus G$. In this case, we choose $g_0 \in \div{G}\setminus G$ such that $g_0$ is a limit point of $G$. We may assume that $g_0$ is a right-sided limit point, as otherwise we can replace it by $-g_0$. Let $x \in L \setminus k\pow{G^{\geq0}}$, i.e.\ $w(x) < 0$. Since $g_0$ is a right-sided limit point of $G$ in $\div{G}$, the interval $(g_0, g_0 - w(x))\subseteq \div{G}$ contains some element $g_1 \in G$. Thus, $g_1 > g_0$ but $w(x) + w(t^{g_1}) = w(x) + g_1 < g_0$. This shows that $x$ does not fulfil condition \eqref{eq:cond1}, whence $x \notin \OO_s$. We thus obtain $\OO_s \subseteq k\pow{G^{\geq0}}$.
		
		Now suppose that $k$ is not real closed and has a limit point $a$ in $\rc{k}\setminus k$. We may assume that $a$ is a left-sided limit point, as otherwise we can replace it by $-a$. Then $D_a' = \{x \in L \mid a- 1 < x < a\}$ consists exactly of the elements of the form $b + r$, where $b\in k$ such that $ a-1 < b < a$ and $r \in k\pow{G^{>0}}$. In other words, $D_a' = I + k\pow{G^{>0}}$, where $I$ is the convex set $(a-1,a)$ in $k$. Note that $I$ is non-empty, as $a$ is a left-sided limit point of $k$. Let $A_a'$ be the $\Lor$-definable set $\setbr{x \in L^{\geq 0} \mid x + D_a' \subseteq D_a'}$. Since $k\pow{G^{>0}}$ is closed under addition, we have $k\pow{G^{>0}} + D_a' \subseteq D_a'$. Thus, $k\pow{G^{>0}}^{\geq 0} \subseteq A_a'$. For the other inclusion, let $x \in L^{\geq0} \setminus k\pow{G^{>0}}$, i.e.\ $w(x) \leq 0$ and $x \geq 0$. If $w(x) > 0$, then $x + b \notin D_a'$ for any $b \in I$. Thus, $x \notin A_a'$. Suppose that $w(x) = 0$. Then $x$ is of the form $c + r$ with $c \in k^{>0}$ and $r \in k\pow{G^{>0}}$. If $c \geq 1$, then $x + b \notin D_a'$ for any $b \in I$, whence $x \notin A_a'$. If $c < 1$, let $b \in k\cap (a-c,a)$, which exists, as $a$ is a left-sided limit point of $k$. Then $x + b = (c+b) + r > a + r$. Thus, $x + b \notin D_a'$ and $x \notin A_a'$. 
		Hence, we have shown that $A_a' \subseteq k\pow{G^{>0}}^{\geq0}$. 
		
		Now $(-A_a' \cup A_a') = k\pow{G^{>0}}$ is the maximal ideal of the valuation ring $k\pow{G^{\geq 0}}$. Thus, the valuation ring $k\pow{G^{\geq 0}} = \{x \in L \mid x(-A_a' \cup A_a') \subseteq (-A_a' \cup A_a')\}$ is $\Lor$-definable.
		
		Now for any of the three cases, there is an $\Lor$-formula $\varphi(x,\ul{y})$ (where in the cases (\ref{lem:defvalnew:1}) and (\ref{lem:defvalnew:2}) $\ul{y}$ is just one free variable) such that $(L,<,w) \models \exists \ul{y}\forall x \ (\varphi(x,\ul{y}) \leftrightarrow w(x) \geq 0)$. By elementary equivalence, there is some $\ul{b} \in K$ such that $(K,<,v) \models \forall x \ (\varphi(x,\ul{b}) \leftrightarrow v(x) \geq 0)$. In other words, $\varphi(x,\ul{b})$ defines $v$ in $K$, as required.
	\end{proof}
	
	\Autoref{lem:defvalnew} is not a full characterisation of all $\Lor$-definable henselian valuations on an ordered field. Indeed, we can choose $K = \R\pow{G}$ for $G$ as in \Autoref{ex:conversedivclosed}. Then $\vmin = v_0$ (see page~\pageref{def:v0}) satisfies neither of the conditions of \Autoref{lem:defvalnew}, but $v_0$ is even $\Lr$-definable by \Autoref{fact:thm44} for $p=2$. Moreover, not every ordered henselian valued field which is not dense in its real closure satisfies one of the conditions of \Autoref{thm:defval}. For instance, $(\R\pow{G},<,\vmin)$, where $G$ is as in \Autoref{prop:hahnnoarch}, is not dense in its real closure by \Autoref{prop:hahnclosed} and does not satisfy any of the conditions of  \Autoref{thm:defval}. 
	
	\begin{corollary}\thlabel{cor:kdensedef}
		Let $(K,<)$ be an ordered field and $v$ a henselian valuation on $K$. Suppose that $Kv$ is not real closed but dense in $\rc{Kv}$.
		Then $v$ is $\Lor$-definable in $K$.
	\end{corollary}
	
	\begin{proof}
		This follows immediately from \Autoref{lem:defvalnew}~(\ref{lem:defvalnew:3}), as any point in $\rc{Kv}\setminus Kv$ is a limit point of $Kv$ in $\rc{Kv}$.
	\end{proof}
	
	\begin{remark}\thlabel{rmk:defthm}
		\begin{enumerate}[wide, labelwidth=!, labelindent=6pt]
			\item\thlabel{rmk:defthm:1} Since any archimedean ordered field is dense in its real closure, a special case of \Autoref{cor:kdensedef} is the following: Let $(K,<)$ be a non-archimedean ordered field. Suppose that the natural valuation $\vnat$ is henselian on $K$ and that $(K\vnat,<)$ is not real closed. Then $\vnat$ is $\Lor$-definable in $K$.
			
			\item \Autoref{lem:defvalnew}~(\ref{lem:defvalnew:2}) implies a similar version of \Autoref{cor:kdensedef} if $vK$ is non-divisible and dense in $\div{vK}$. However, we will see in \Autoref{sec:lrdef} that under this condition we already have that $v$ is $\Lr$-definability without parameters.
			
			\item The proof of \Autoref{thm:defval} also shows the following:	Let $(K,<)$ be an ordered field and $v$ a non-trivial henselian valuation on $K$. Suppose that $vK$ is non-divisible. Then there exists a non-trivial $\Lor$-definable coarsening of $v$ which is definable by an $\Lor$-formula with one parameter.
		\end{enumerate}
	\end{remark}
	
	\subsection*{ Comparison to $\Lr$-definability.}
	 We will give a brief account of the known $\Lr$-definability result of henselian valuations in the case that the value group is regular (cf.\ \cite[Theorem~4]{hong}) and compare this to \Autoref{lem:defvalnew}.
	
	\begin{fact}\emph{(See \cite[Theorem~4]{hong}.)}\thlabel{fact:hong}
		Let $K$ be a field and $v$ a henselian valuation on $K$. Suppose that $vK$ is regular and non-divisible. Then $v$ is parameter-free $\Lr$-definable in $K$.
	\end{fact}\label{page:hong}

	By \Autoref{prop:regimpliesdense}, we obtain the following.

	\begin{corollary}\thlabel{cor:hong}
		Let $(K,<)$ be an ordered field and $v$ a henselian valuation on $K$ such that $vK$ is non-divisible but dense in $\div{vK}$. Then $v$ is parameter-free $\Lr$-definable in $K$.
	\end{corollary}
	
	\Autoref{ex:regnew}~(\ref{ex:regnew:1}) shows that there are discretely ordered abelian groups which are not regular. \Autoref{ex:regnew}~(\ref{ex:regnew:2}) exhibits a densely ordered abelian group $G$ which is not regular but has limit points in $\div{G} \setminus G$. This shows that there are ordered fields such that the cases (\ref{lem:defvalnew:1}) and (\ref{lem:defvalnew:2}) of \Autoref{lem:defvalnew} are not already covered by \Autoref{fact:hong}. 
	
	\subsection{Almost real closed fields}\label{sec:arc}\label{sec:lrdef} 
	
	The class of almost real closed fields has first been studied systematically with respect to algebraic and model theoretic properties in \cite{delon}. Moreover,  \cite[Theorem~4.4]{delon} completely characterises all $\Lr$-definable henselian valuations in almost real closed fields. In the following, we will compare $\Lr$- and $\Lor$-definability of henselian valuations in almost real closed fields.
	
	\begin{definition} \thlabel{def:arc}
		Let $(K,<)$ be an ordered field, $G$ an ordered abelian group and $v$ a henselian valuation on $K$. We call $K$ an \textbf{almost real closed field (with respect to $v$ and $G$)} if $Kv$ is real closed and $vK = G$. 
	\end{definition}
	
	Depending on the context, we may simply say that $(K,<)$ is an almost real closed field without specifying the henselian valuation $v$ or the ordered abelian group $G = vK$. 
	In \cite{delon}, almost real closed fields are considered as pure fields, i.e.\ as structures in the language $\Lr$. Note that due to the Baer--Krull Representation Theorem (cf.\ \cite[page~37~f.]{engler}), any such field admits (possibly several distinct) orderings. We, however, consider almost real closed fields as ordered fields with a fixed order.
	By \cite[Proposition~2.9]{delon}, any convex valuation on an almost real closed field is already henselian. We thus make no distinction between convex and henselian valuations on almost real closed fields.

	



	Let $(K,<)$ be an almost real closed field. We denote by $V(K)$ the set of all henselian valuations on $K$ with real closed residue field and {  by $v_0$\label{def:v0} the minimum of $V(K)$, i.e.\ the coarsest valuation in $V(K)$. This exists by \cite[Proposition~2.1~(iv)]{delon}. Moreover, \cite[Proposition~2.1~(ii)]{delon} implies that $\vnat$ is the maximum of $V(K)$, as the natural valuation is the finest convex and thus henselian valuation on $K$.}
	By the remarks in \cite[page~1147~f.]{delon}, $v_0$ is the only possible $\Lr$-definable henselian valuation in $V(K)$. Also in the language $\Lor$, there is at most one definable valuation in $V(K)$.
	
	\begin{proposition}\thlabel{prop:rcuniquedef}
		Let $(K,<)$ be an almost real closed field and $v \in V(K)$. Suppose that $v$ is $\Lor$-definable in $K$. Then $v$ is the only $\Lor$-definable valuation in $V(K)$. 
	\end{proposition}	
	
	\begin{proof}
		By \Autoref{fact:ake}, we have $$(K,<,v) \equiv (\R\pow{vK},<,\vmin).$$ Since $v$ is $\Lor$-definable in $K$, there exists an $\Lor$-formula $\varphi(x,\ul{y})$ such that $$K \models \exists \ul{y} \forall x\  (\varphi(x,\ul{y}) \leftrightarrow v(x) \geq 0).$$ By elementary equivalence, there exists $\ul{b} \in \R\pow{vK}$ such that $$\R\pow{vK} \models \forall x\  (\varphi(x,\ul{b}) \leftrightarrow \vmin(x) \geq 0).$$ Hence, $\vmin$ is $\Lor$-definable in $\R\pow{vK}$. 
		
		Let $\psi(x,\ul{y})$ be an $\Lor$-formula and $\ul{c} \in K$ such that $\psi(x,\ul{c})$ defines a convex valuation $w$ in $K$. Assume, for a contradiction, that $w$ is strictly finer than $v$, i.e.\ $\OO_w \subsetneq \OO_v$. This implies
		$$(K,<,v) \models \forall x \ (\psi(x,\ul{c}) \to v(x) \geq 0) \wedge \exists z\ ( \neg \psi(z,\ul{c}) \wedge v(z) \geq 0 ).$$
		By elementary equivalence, there is some $\ul{c}' \in \R\pow{vK}$ such that
		$\psi(x,\ul{c}')$ defines a convex valuation $w'$ in $\R\pow{vK}$ with $\OO_{w'} \subsetneq \OO_{\vmin}$. This contradicts the fact that $\vmin$ is the finest convex valuation on $\R\pow{vK}$. Hence, $v$ is the finest $\Lor$-definable convex valuation in $K$.

		Let $v' \in V(K)$ be $\Lor$-definable. Arguing as above, $v'$ is the finest $\Lor$-definable convex valuation on $K$. This gives us $v' = v$, as required.
	\end{proof}
	
	\subsection*{Comparison of $\Lr$- and $\Lor$-definability of henselian valuations in almost real closed fields.}
	Let $p$ be a prime number. A valuation $v$ on $K$ is called {$p$-Kummer henselian} if Hensel's Lemma holds for polynomials of the form $x^p-a$ for $a \in \OO_v$. A field $L$ is called {$p$-euclidean} if $L = \pm L^p$. Let $V_p(K)$ be the set of all $p$-Kummer henselian valuations of $K$ with $p$-euclidean residue field. Denote by $v_p$ the minimum of $V_p(K)$ (cf.\ \cite[page~1126]{delon}).
	
	\begin{fact}\thlabel{fact:thm44}\emph{(See \cite[Theorem~4.4]{delon}.)}
		Let $(K,<)$ be an almost real closed field and let $v$ be a henselian valuation on $K$. Then $v$ is $\Lr$-definable in $K$ if and only if { $\vnat(\OO_v\setminus\MM_v)$ is $\Log$-definable in $\vnat K$} and $v \leq v_p$ for some prime $p$. Moreover, $v_0$ is $\Lr$-definable if and only if there is a prime $p$ such that $v_0K$ has no non-trivial convex $p$-divisible subgroup. 
	\end{fact}

	Recall that $v_0$ is the only possible $\Lr$-definable valuation in $V(K)$. If the ordering on an almost real closed field $(K,<)$ is $\Lvf$-definable for $v = v_0 \in V(K)$, then we obtain a complete characterisation of $\Lor$-definable convex valuations in $K$.
	
	\begin{lemma}\thlabel{lem:orderingdef2}
		Let $(K,<)$ be an ordered field and let $v$ be a henselian valuation on $K$ such that $Kv$ is $2$-euclidean (i.e.\ root-closed for positive elements) and $vK$ is $2$-divisible. Then the ordering $<$ is parameter-free $\Lvf$-definable in $K$.
		In particular, if $v$ is $\Lr$-definable in $K$, then any $\Lor$-definable subset of $K$ is already $\Lr$-definable.
	\end{lemma}
	
	\begin{proof}
		Let $k=Kv$ and $G=vK$. Consider the $\Lvf$-formula $\varphi(x)$ given by $$x = 0 \vee \exists y\ v(x-y^2) > v(x).$$
		We will show that for any $a \in k\pow{G}$, the formula $\varphi(a)$ holds if and only if $a \geq 0$. 
		Let $a = a_gt^g + s \in k\pow{G}^\times$, where $a_g \in k^\times$, $s \in k\pow{G^{>g}}$ and $g = \vmin(a)$.
		Suppose that $\varphi(a)$ holds. Then there exists $y \in K^\times$ such that $\vmin(x-y^2) > g$. Hence, $a_g = y_g^2 > 0$, where $y_g$ is the coefficient of the monomial $t^g$ in $y$. Thus, $a > 0$.
		Now suppose that $a > 0$. Let $y = \sqrt{a_g}t^{g/2}$. Then $\vmin(a-y^2) = \vmin(s) > g = \vmin(a)$.
		
		By \Autoref{fact:ake}, $(K,<,v) \equiv (Kv\pow{vK},<, \vmin)$. Hence, we obtain $K \models \forall x \ (x\geq 0 \leftrightarrow \varphi(x))$.
	\end{proof}
	
	\begin{proposition}\thlabel{prop:charlrdef}
		Let $(K,<)$ be an almost real closed field. Suppose that $v_0$ is $\Lr$-definable and that $v_0K$ is $2$-divisible. Let $w$ be a  valuation on $K$. If $w$ is $\Lor$-definable, then it is $\Lr$-definable.
	\end{proposition}
	
	\begin{proof}
		Since $v_0K$ is real closed, it is $2$-euclidean. By \Autoref{lem:orderingdef2}, any $\Lor$-definable valuation on $K$ is already $\Lr$-definable.
	\end{proof}
	
	\begin{remark}
		We obtain the following characterisation of $\Lor$-definable convex valuations in almost real closed fields with respect to a $2$-divisible value group:
		Let $(K,<)$ be an almost real closed field. Suppose that the value group $v_0K$ is $2$-divisible and, for some prime $p$, it has no non-trivial $p$-divisible convex subgroup. Let $v$ be a convex (and thus henselian) valuation on $K$. Hence, by \Autoref{prop:charlrdef} and \Autoref{fact:thm44}, $v$ is $\Lor$-definable in $K$ if and only if { $\vnat(\OO_v\setminus\MM_v)$} is $\Log$-definable in { $\vnat K$} and $v \leq v_p$ for some prime $p$.
	\end{remark}
	
	\Autoref{prop:charlrdef} shows in particular that for an almost real closed field, if $v_0$ is $\Lr$-definable, then so is any $\Lor$-definable henselian valuation. The final result of this sections shows that any henselian valuation in an almost real closed field satisfying the hypothesis of \Autoref{thm:defval} is already $\Lr$-definable. Note that any discretely ordered abelian group does not have a non-trivial $n$-divisible convex subgroup for any $n \geq 2$.
	
	\begin{proposition}\thlabel{prop:lorthenlr}
		Let $(K,<)$ be an almost real closed field with respect to a henselian valuation $v$ such that either $vK$ is discretely ordered or $vK$ is not closed in $\div{vK}$. Then $v = v_0$ and $v_0$ is $\Lr$-definable.
	\end{proposition}
	
	\begin{proof}
		Recall that $v_0$ is the unique henselian valuation on $K$ such that $Kv_0$ is real closed and $v_0K$ has no non-trivial divisible convex subgroup. Moreover, $v_0$ is $\Lr$-definable if and only if for some prime $p$, there is no $p$-divisible non-trivial convex subgroup of $v_0K$. If $G = vK$ is discretely ordered, both conditions are satisfied, and thus $v = v_0$ is $\Lr$-definable. If $G$ is not closed in $\div{G}$, then by \Autoref{cor:convdivlimit}, $G$ has no non-trivial convex subgroup, whence $v = v_0$. Moreover, by \Autoref{prop:convdivlimit} there is a prime $p$ such that $G$ has no non-trivial $p$-divisible convex subgroup, whence $v_0$ is $\Lr$-definable.
	\end{proof}


	\section{Applications to definable valuations on strongly NIP ordered fields}\label{sec:snip}
	
	{ The study of $\Lr$-definable henselian valuations in strongly NIP fields has been motivated by a conjecture due to Shelah--Hasson stating as follows (cf.\ \cite[page~63]{shelah} and \cite[page 820]{dupont}).
	\begin{conjecture}[Shelah--Hasson]\thlabel{conj:shelahhasson}
		Any { infinite} strongly NIP field is either real closed, algebraically closed or admits a non-trivial $\Lr$-definable henselian valuation.
	\end{conjecture}
	In the past few years, a considerable amount of research activity dealing with this conjecture could be observed, in particular since it was proved by Johnson for the special case of dp-minimal fields (cf.\ \cite[Theorem~1.6]{johnson}) and more recently for dp-finite fields (cf.\ \cite[Theorem~1.2]{johnson2}).
	 Specialised to ordered fields, \autoref{conj:shelahhasson} can be reformulated as follows:\\
	 
	 	\emph{Any strongly NIP ordered field which is not real closed admits a non-trivial $\Lor$-definable henselian valuation.}
	\\
	\\
 	 (Note that we replaced $\Lr$-definabiliy by $\Lor$-definability, as definability in the language of ordered rings is more natural in the context of ordered fields.) Again, this has already been verified for the dp-minimal case (cf.\ \cite[Corollary~6.6]{jahnke}) but is still open for general strongly NIP ordered fields.
 	 
 	 In \cite{krapp}, it is shown that the statement above is equivalent to the following.\footnote{As mentioned in the introduction, this conjecture is the main subject of a separate publication \cite{krapp2}.}
 	 \begin{conjecture}[Shelah--Hasson specialised to ordered fields]\thlabel{conj:snipordered}
 	 	 
 	 Any strongly NIP ordered field is almost real closed.
  	 \end{conjecture} 
   
     Comparing the original formulation of the Shelah--Hasson Conjecture to the formulation in \autoref{conj:snipordered}, it is tempting to ask the following question.
     
     \begin{question}[Strengthening of \autoref{conj:snipordered}]\thlabel{qu:strengthening}
     	Any strongly NIP ordered field is almost real closed \textbf{with respect to} an $\Lor$-definable henselian valuation. 
     \end{question}
   	
   	 
   	 \autoref{thm:negativeanswer} below provides a negative answer to \autoref{qu:strengthening}.
   	 }
	
	All notions on strongly NIP theories can be found in \cite{simon}. For the purpose of understanding how \Autoref{fact:thm44} can be applied to { strongly} NIP ordered fields, the following two results are helpful.
	
	\begin{fact}{\rm \cite[Theorem~1]{halevi2}}\thlabel{fact:snipgp}
		Let $G$ be an ordered abelian group. Then the following are equivalent:
		\begin{enumerate}[wide, labelwidth=!, labelindent=6pt]
			\item $G$ is strongly NIP.
			
			\item $G$ is elementarily equivalent to a Hahn sum of archimedean ordered abelian groups $\coprod_{i \in I} G_i$, where for every prime $p$, we have
			$|\{i \in I \mid pG_i \neq G_i \}| < \infty$,
			and for any $i \in I$, we have
			$|\{p \text{ prime} \mid [G_i:pG_i]=\infty \}| < \infty$.
		\end{enumerate}
	\end{fact}

	\begin{lemma}{\rm \cite[Lemma~6.11]{krapp}}\thlabel{lem:powstrongnip}
		Let $G$ be a strongly NIP ordered abelian group. Then the almost real closed field $(\R\pow{G}, <)$ is strongly NIP.
	\end{lemma}

	By \Autoref{fact:snipgp}, the following are examples of strongly NIP non-divisible ordered abelian groups.
	
	\begin{example}\thlabel{ex:snipdensenotdense}
		\begin{enumerate}[wide, labelwidth=!, labelindent=6pt]
			\item \label{ex:snipdensenotdense:1} Let 
			$$ G_1 = \setbr{\left. \frac{a}{p_1^{i_1}\ldots p_m^{i_m}} \ \right|m \in \N;  { a\in\Z,i_1,\ldots,i_m \in \N_0; p_1,\ldots,p_m\geq 3 \text{ prime}}}\!.$$
			$G_1$ is $p$-divisible for any prime $p\geq 3$. Thus,  
			$|\{p \text{ prime} \mid [G_1:pG_1]=\infty \}| \leq 1$. This implies that $G_1$ is strongly NIP. Moreover, it is dense in its divisible hull $\Q$ but not divisible, as $\frac 1 2 \notin G_1$.
			
			\item \label{ex:snipdensenotdense:2} Let $2 = p_0 < p_1 < \ldots$ be a complete list of prime numbers. For every $n \in \N$, let $$B_n = \setbr{\left.\frac{a}{p_{i_1}^{m_1}\cdot\ldots\cdot p_{i_k}^{m_k}} \ \right| \ k \in \N; i_1,\ldots,i_k \in \N_0\setminus{\{n\}}; { a\in \Z,m_1,\ldots,m_k \in \N_0}}\!.$$ In other words, $B_n$ is the smallest subgroup of $\Q$ which is $p_i$-divisible for any $i \in \N_0\setminus{\setbr{n}}$ (so in particular $2$-divisible) but not $p_n$-divisible. Then $G_2 = \coprod_{n \in \N} B_n$ is strongly NIP, as for any $i \in \N_0$, we have $|\{n \in \N \mid p_iB_n \neq B_n \}|  \leq  1$,
			and for any $n \in \N$, we have
			$|\{p \text{ prime} \mid [B_n:pB_n]=\infty \}| \leq 1$. Note that $G_2$ is $2$-divisible. Moreover, any convex subgroup of $G_2$ is of the form $\coprod_{n \geq k} B_n = H_k$ for some $k \in \N$. Note that $H_k$ is non-divisible but $p_i$-divisible for any $i < k$.
		\end{enumerate}
	\end{example}

	Let $G_1$ and $G_2$ be { the strongly NIP ordered abelian groups defined in} \Autoref{ex:snipdensenotdense}. By \Autoref{lem:powstrongnip}, we obtain strongly NIP almost real closed fields $(\R\pow{G_1},<)$ and $(\R\pow{G_2},<)$. 
	
	\begin{example}
			{ Let $K_1=\R\pow{G_1}$. 
				By \Autoref{prop:lorthenlr}, $v_0=\vnat$ is $\Lr$-definable in $(K_1,<)$, as $G_1$ is dense in its divisible hull. This shows that $(K_1,<)$ is a strongly NIP ordered field which is not real closed but almost real closed with respect to an $\Lr$-definable (and hence $\Lor$-definable) henselian valuation.}
	\end{example}

	\begin{example}\thlabel{thm:negativeanswer}
		{ 
			Let $K_2=\R\pow{G_2}$.
			Since $G_2$ has no convex divisible subgroup, we have $v_0=\vnat$. Thus by \Autoref{fact:thm44}, this valuation is $\Lr$-definable in $K_2$ if and only if there is a prime $p$ such that $G_2$ has no non-trivial convex $p$-divisible subgroup. However, for any $i \in \N$, the non-trivial convex subgroup $H_{i+1}$ of $G_2$ is $p_i$-divisible, and $G_2$ itself is $2$-divisible. This implies that $\vnat$ is not $\Lr$-definable. Moreover, since $G_2$ is $2$-divisible, $\vnat$ is also not $\Lor$-definable by \Autoref{prop:charlrdef}. As $\vnat= v_0$ (see page~\pageref{def:v0}), it is the only valuation with respect to which $K_2$ is almost real closed.\emph{ Thus, $(K_2,<)$ is a { strongly} NIP almost real closed field which is not almost real closed with respect to an $\Lor$-definable henselian valuation.}}
	\end{example}


	

	\section{Open Questions}\label{sec:questions}
	
	We conclude with open questions connected to results throughout this work.
	
	In \Autoref{prop:subgroupdef} we have seen the following dichotomy for ordered abelian groups: Any ordered abelian group $G$ is (exclusively) either dense in $\div{G}$ or admits a proper non-trivial $\Log$-definable convex subgroup. For ordered fields, it is not known whether a similar dichotomy holds.
	Note that any archimedean ordered field $(K,<)$ does not admit a non-trivial valuation, as $\Z$ must be contained in any convex subring of $K$. Hence, any ordered field with an archimedean model does not admit a non-trivial $\Lor$-definable convex valuation. However, we have seen in \Autoref{prop:ordfieldnoarchmodel} that there are non-archimedean ordered fields which are dense in their real closure but do not have an archimedean model. In these ordered fields, it may be possible to find non-trivial $\Lor$-definable convex valuations. Note that these cannot be henselian by \Autoref{cor:nohensval}. In \cite{jahnke} it is not investigated whether the two cases in \Autoref{prop:defval} are exclusive. We pose this as the following question. 
	
	\begin{question}
		Is there an ordered field which is dense in its real closure and admits an non-trivial $\Lor$-definable convex valuation?
	\end{question}
	
	In \Autoref{lem:lemmaerdos} and \Autoref{prop:lemmaerdos} we have seen that any ordered field $K$ of infinite transcendence degree admits a dense transcendence basis $T$ and thus a dense subfield $F = \Q(T)$ such that $\Q \subseteq F$ is regular. Note that a non-archimedean ordered field $K$ with $\td(K) < \aleph_0$ cannot admit a transcendence basis dense in $K$. However, it is still possible that $K$ has a dense subfield $F$ such that $\Q\subseteq F$ is regular. We pose this as a question for a non-archimedean real closed field with transcendece degree $1$.
	
	\begin{question}\thlabel{qu:denseregular}
		Let $K = \rc{\Q(t)}$ ordered by $t > \N$. Is there a dense subfield $F \subseteq K$ such that $\Q \subseteq F$ is regular?
	\end{question}

	We have seen in \Autoref{prop:lorthenlr} that for any almost real closed field with respect to a henselian valuation $v$, if $v$ satisfies the hypothesis of \Autoref{thm:defval}, then it is not only $\Lor$- but already $\Lr$-definable. Although we have shown that \Autoref{thm:defval} generalises known $\Lr$-definability results of henselian valuations (see page~\pageref{page:hong}), we have not provided an example of an ordered field $(K,<)$ and a henselian valuation $v$ on $K$ which is $\Lor$- but not $\Lr$-definable. Since \Autoref{fact:thm44} provides a full characterisation of $\Lr$-definable henselian valuations in almost real closed fields, we will pose the following more specific question.
	
	\begin{question}
		Is there an almost real closed field $(K,<)$ and a henselian valuation $v$ on $K$ such that $v$ is $\Lor$- but not $\Lr$-definable?
	\end{question}

\end{document}